\documentclass[12pt]{article}

\usepackage{amsmath}
\usepackage{amssymb}
\usepackage{amsthm}
\usepackage{amsfonts}
\usepackage{latexsym}
\usepackage{cite}
\usepackage{epsfig}
\usepackage{graphicx}
\usepackage{mathrsfs}
\usepackage{url}
\usepackage{color}
\usepackage{eufrak}
\usepackage{multirow}
\usepackage{hyperref}

\usepackage{scalerel}
\usepackage{tikz}
\usetikzlibrary{svg.path}

\definecolor{orcidlogocol}{HTML}{A6CE39}
\tikzset{
  orcidlogo/.pic={
    \fill[orcidlogocol] svg{M256,128c0,70.7-57.3,128-128,128C57.3,256,0,198.7,0,128C0,57.3,57.3,0,128,0C198.7,0,256,57.3,256,128z};
    \fill[white] svg{M86.3,186.2H70.9V79.1h15.4v48.4V186.2z}
                 svg{M108.9,79.1h41.6c39.6,0,57,28.3,57,53.6c0,27.5-21.5,53.6-56.8,53.6h-41.8V79.1z M124.3,172.4h24.5c34.9,0,42.9-26.5,42.9-39.7c0-21.5-13.7-39.7-43.7-39.7h-23.7V172.4z}
                 svg{M88.7,56.8c0,5.5-4.5,10.1-10.1,10.1c-5.6,0-10.1-4.6-10.1-10.1c0-5.6,4.5-10.1,10.1-10.1C84.2,46.7,88.7,51.3,88.7,56.8z};
  }
}

\newcommand\orcidicon[1]{\href{https://orcid.org/#1}{\mbox{\scalerel*{
\begin{tikzpicture}[yscale=-1,transform shape]
\pic{orcidlogo};
\end{tikzpicture}
}{|}}}}

\usepackage{hyperref} 

\usepackage{longtable}

\makeatletter
\@addtoreset{equation}{section}
\makeatother

\makeatletter
\@addtoreset{table}{section}
\makeatother

\newtheorem{theorem}{Theorem}[section]
\newtheorem{lemma}[theorem]{Lemma}

\theoremstyle{definition}
\newtheorem{definition}[theorem]{Definition}
\newtheorem{notation}[theorem]{Notation}
\newtheorem{remark}[theorem]{Remark}

\newcommand{\Vc}{\mathcal{V}}

\newcommand{\PG}{\mathrm{PG}}

\newcommand{\zb}{\mathbf{0}}

\newcommand{\F}{\mathbb{F}}

\newcommand{\Bs}{\mathscr{B}}
\newcommand{\Ps}{\mathscr{P}}

\newcommand{\Ab}{\boldsymbol{A}}

\newcommand{\D}{\boldsymbol{D}}
\newcommand{\Hb}{\boldsymbol{H}}
\newcommand{\hb}{\boldsymbol{h}}
\newcommand{\Ib}{\boldsymbol{I}}
\newcommand{\Mb}{\boldsymbol{M}}
\newcommand{\Ub}{\boldsymbol{U}}
\newcommand{\ub}{\boldsymbol{u}}
\newcommand{\pib}{\boldsymbol{\pi}}
\newcommand{\Wb}{\boldsymbol{W}}
\newcommand{\Hcb}{\boldsymbol{\mathcal{H}}}

\newcommand{\wb}{\boldsymbol{w}}

\newcommand{\xb}{\boldsymbol{x}}
\newcommand{\yb}{\boldsymbol{y}}
\newcommand{\zbl}{\boldsymbol{z}}

\newcommand{\ov}{\overline{v}}
\newcommand{\oc}{\overline{c}}

\newcommand{\ozero}{\overline{0}}

\newcommand{\T}{\text}
\newcommand{\db}{\displaybreak[3]}

\begin{document}

\title{
New upper bounds for binary linear covering codes}
\date{}
\maketitle
\begin{center}
{\sc Alexander A. Davydov \orcidicon{0000-0002-5827-4560}}\\
 {\sc\small Kharkevich Institute for Information Transmission Problems}\\
 {\sc\small Russian Academy of Sciences,
Moscow, 127051, Russian Federation}\\
 \emph{E-mail address:} alexander.davydov121@gmail.com\medskip\\
 {\sc Stefano Marcugini  \orcidicon{0000-0002-7961-0260} and
 Fernanda Pambianco  \orcidicon{0000-0001-5476-5365}}\\
 {\sc\small Department of  Mathematics  and Computer Science,  Perugia University,}\\
 {\sc\small Perugia, 06123, Italy}\\
 \emph{E-mail address:} \{stefano.marcugini, fernanda.pambianco\}@unipg.it
\end{center}

\textbf{Abstract.}
The length function $\ell_2(r,R)$ is the smallest
length of a binary linear code with codimension (redundancy) $r$ and covering radius $R$.
We obtain the following new upper bounds on $\ell_2(r,R)$, which yield a decrease
$\Delta(r,R)$ compared to the best previously known upper bounds:
\begin{align*}
&R=2,\,r=2t,\,r=18,20,\T{ and }r\ge28,\,\ell_2(r,2)\le26\cdot2^{r/2-4}-1;\,\Delta(r,2)=2^{r/2-4}.\\
&R=3,\,r=3t-1,\,r=26\T{ and }r\ge44,\,\ell_2(r,3)\le819\cdot2^{(r-26)/3}-1;\,\Delta(r,3)=2^{(r-23)/3}.\\
&R=4,\,r=4t,\,r=40\T{ and }r\ge68,\,\ell_2(r,4)\le2943\cdot2^{r/4-10}-1;\,\Delta(r,4)=2^{r/4-10}-1.
\end{align*}
To obtain these bounds we construct new infinite code families, using distinct versions of the $q^m$-concatenating constructions of covering codes; some of these versions are proposed in this paper. We also introduce new useful partitions of column sets of parity check matrices of some codes. The asymptotic covering densities $\overline{\mu}(2)\thickapprox1.3203$,
$\overline{\mu}(3)\thickapprox1.3643$, $\overline{\mu}(4)\thickapprox2.8428$, provided by the codes of the new families, are smaller than the known ones.



\section{Introduction, notations, and background}\label{sec1:Intro}
\subsection{Covering codes, covering density, the length function}\label{subsec11:CovCodes}
Let $\F_{q}$ be the \emph{Galois field} with $q$ elements, $\F_q^*=\F_{q}\setminus\{0\}$. Let $\triangleq$ denote the sign ``equality by definition''. The \emph{$n$-dimensional vector space over $\F_{q}$} is denoted by $\F_{q}^{\,n}$. Let $\F_q^{\,m\times n}$ be the set of $m \times n$ matrices with entries from $\F_q$. Let $\#S$ be the cardinality of a set $S$.  Let $t$-set (resp. $t$-subset) be a set (resp. subset) of cardinality~$t$. We use lowercase letters, e.g. $v$, for scalars, and overlined lowercase letters, e.g. $\ov$, for vectors. Let $\ozero$ be the zero vector. Let $tr$ denote transposition. We denote by $\Ib_{r}$ the identity matrix of order $r$ whose leftmost column is $(10\ldots0)^{tr}$.

The \emph{Hamming weight} $\T{wt}(\ov)$ of a vector $\ov\in\F_{q}^{\,n}$ is the number of its non-zero coordinates.
The \emph{Hamming distance} $d(\ov,\ov')$ between two vectors $\ov$, $\ov'$ is the number of coordinates in which they differ.
The \emph{Hamming ball} of radius $\varrho$ with center $\ov\in \F_{q}^{\,n}$ is the set $\{\ov'\,|\,\ov'\in \F_{q}^{\,n},~d(\ov,\ov')\leq \varrho\}$, its volume $V_{\varrho,n,q}$ is given in \eqref{eq11:HamBall}.
\begin{align}
& V_{\varrho,n,q}=\sum_{i=0}^\varrho\binom{n}{i}(q-1)^i;~
 V_{\varrho,n,2}=\sum_{i=0}^\varrho\binom{n}{i}.\label{eq11:HamBall}
\end{align}

A \emph{$q$-ary linear code} $C$ of length $n$ and dimension $k$ is  a $k$-dimensional subspace of $\F_{q}^{\,n}$; it is denoted as $[n,k]_q$ code, and $\#C=q^k$. If $q=2$, we refer to $C$ as a \emph{binary code}. Let $r\triangleq n-k$ be \emph{codimension} (or \emph{redundancy}) of an $[n,k]_q$ code. A value, say ``$\bullet$'', associated with a code $C$ is denoted by $``\bullet(C)"$, but if $C$ is clear from the context or $C$ is not relevant we simply write $``\bullet"$.

A vector $\oc\in C$ is a \emph{codeword} of $C$. The \emph{minimum distance} (or simply \emph{distance}) $d$ of the code $C$ is the minimum Hamming distance between any pair of codewords.
The \emph{distance of a vector $\ov\in\F_{q}^{\,n}$
 from a code $C$} is defined by $d(\ov,C)\triangleq \min\limits_{\oc\in C} d(\ov,\oc)$.

An $(n-k)\times n$ matrix $\Hb\in\F_q^{\,(n-k)\times n}$ is a \emph{parity check matrix} of a linear $[n,k]_q$ code $C$ if, for any codeword $\oc\in C$, we have $\oc\times\Hb^{tr}=\ozero$.

\begin{definition} \label{Def11:CoverRad}
A linear $[n,k]_q$ code $C$ has covering radius $R$  if any of the following equivalent conditions  holds:

\textbf{(i)} The value $R$  is the smallest integer such that the space $\F_{q}^{\,n}$ is covered by
Hamming balls of radius $R$ centered at the codewords of $C$, i.e. $ R(C)\triangleq\max\limits_{\ov\in\F_{q}^{\,n}}d(\ov,C)$.

\textbf{(ii)}
Every vector in $\F_{q}^{\,n-k}$ can be expressed as a linear combination
of at most $R$ columns of an $(n-k)\times n$ parity check matrix of $C$, and $R$ is the smallest value with this property.  For $q=2$, we say ``sum of columns'' instead of ``linear combination.''
\end{definition}

An $[n,k]_q$ code with minimum distance $d$ and covering radius $R$ is denoted $[n,k,d]_qR$,
where $d$ and $R$ can be omitted when they are not relevant or clear from the context.

For an introduction to coding theory, see \cite{libroBierbrauer,HufPless,MWS,Roth} and the references therein.

\begin{definition} \label{Def11:CoverDensReg}
The \emph{covering density} $\mu(n,R,C)$  of an $[n,k]_qR$-code $C$ is defined as the ratio of the total volume of all $q^k V_{R,n,q}$ Hamming balls of radius $R$ centered at the codewords to the total volume $q^n$ of the space~$\F_{q}^{\,n}$, where $V_{R,n,q}$ is given in \eqref{eq11:HamBall}. We have
\begin{align}
&\mu(n,R,C)\triangleq\frac{q^k V_{R,n,q}}{q^n}=\frac{1}{q^{n-k}}\sum_{i=0}^R\binom{n}{i}(q-1)^i,~q\ge2;\notag\db\\
&\mu(n,R,C)\triangleq\frac{1}{2^{n-k}}\sum_{i=0}^R\binom{n}{i},~q=2.\label{eq11:CovDensity}
\end{align}
By Definition \ref{Def11:CoverRad}, $\mu(n,R,C)\ge1$.
We may write $\mu(C)$, $\mu(n-k)$, $\mu(r)$, or simply $\mu$ if the other parameters are clear by the context.
\end{definition}

  Let $U$ be an infinite family of codes with covering radius $R$, and let $U_n$ be a code in $U$ with length $n$. For the family $U$ we consider the asymptotic parameter \cite{Dav90PIT,GabDavTombR=2}:
\begin{equation}\label{eq11:liminfdens}
  \overline{\mu}(R,U)\triangleq\liminf_{n\rightarrow\infty,\,U_n\in U}\mu(n,R,U_n).
\end{equation}
We may write $\overline{\mu}(R)$  when $U$ is defined by the context.

The covering quality of a code is better if its covering density is smaller.
For fixed $n-k$ and $R$, the covering density $\mu$  decreases as $n$ decreases.

The covering problem for codes is to find codes with small covering radius
with respect to their lengths and dimensions. Codes investigated
 from this perspective are called \emph{covering codes} (as opposed to error-correcting codes).
When covering radius and codimension
(redundancy) are fixed, the covering problem becames finding codes with the smallest possible length and/or obtaining good upper bounds on the length.

\begin{definition}\label{def1:length function}
The \emph{length function} $\ell_q(r,R)$ is the smallest
length of a $q$-ary linear code with codimension (redundancy) $r$ and covering radius $R$.
\end{definition}

For an introduction to covering radii, coverings of vector Hamming spaces over finite fields, and covering codes, see
\cite{GrahSlo,CKMS-survey1985,Handbook-coverings,CHLL-bookCovCod,DGMP-AMC,HufPless,Delsarte} and the references therein. Additional  results (including bounds, parameters, constructions, and generalizations), related to covering codes are discussed in numerous works, see the online bibliography \cite{LobstBibl}. For example, we note
\cite{CLLM-survey1985-1994,Graismer-2024,BrPlWi,CohenFranklBounds,CLS-1986,Dav90PIT,DavParis,%
DDL-ACCT2-1990,DDL-ACCT3-1992,DDL-IEEE,HonkLits1996,%
Giul2013Survey,Janwa,KilbySloane1987-II,OstKaikUpBndBin1998,Struik1994R=2-3Bin,Struik,PlanFuncCovCodWuPan2025,%
DougJanCovRadCalc1991,LobsPleRevisTab1994}.

\subsection{Connection of covering codes with other areas of theory and practice}\label{subsec12:Connection}
The study of covering codes is a classical combinatorial problem. Covering codes are connected  to many areas of information theory, combinatorics, and practical applications.  For example, see \cite[Section 1.2]{CHLL-bookCovCod}, where applications include:  problems of data compression, decoding errors and erasures, football pools \cite{HamalHonLitsOstFootbPool,HamalFootbPool}, write-once memories, Berlekamp-Gale game, and Cayley graphs \cite{Cayley}. Covering codes can also be applied to reducing switching noise and to steganography \cite[Chapter 14]{libroBierbrauer}, \cite{BierbStegan,GalKaba,GalKabaPIT}. Linear codes of covering radius 2 and codimension 3 are relevant, in particular, for defining sets of block designs \cite{Bo-Sz-Ti} and for the degree/diameter problem in graph theory \cite{Cayley,KKKPS}.  Covering codes can be used in databases
\cite{PartSumQuer}, in constructions of identifying codes \cite{identif2}, for solving the learning parity with noise (LPN problem) \cite{LPN-CovCod}, in analysis of blocking switches \cite{swithc}, in reduced representations of logic functions \cite{Bulev,StanAstMor-book}, in list decoding of error correcting codes \cite{ListDecCovCod-2022}, in cryptography \cite{MenNewCrypCovCod}. There are also interesting connections between covering codes and following combinatorial problems: a popular game puzzle, called ``Hats-on-a-line'' \cite{Hats-on-line,LenSerHat2002}; multiple coverings of the farthest-off points or deep holes \cite{BDGMP-MultCovFurth,KaipaDeepHoles}; factorial designs \cite{ChengTang-FactDesigns}. In \cite{NavSamor-LPbound&CovArgum2009},  linear programming bounds for codes via a covering argument are considered.

\subsection{The one-to-one correspondence between covering codes and saturating sets in projective spaces}\label{subsec13:CovCode=SatSet}
Let $\PG(N,q)$ be the $N$-dimensional projective space over the Galois field~$\mathbb{F}_q$.
  A point set $S\subseteq\PG(N,q)$ is $\rho$-\emph{saturating} if every point $A\in\PG(N,q)$ can be written as a linear combination of at most $\rho+1$ points of $S$ and $\rho$ is the smallest value with this property.
 Saturating sets are also called ``saturated sets'', ``spanning sets'', and ``dense sets''.

Let $s_q(N,\rho)$ be the smallest size of a $\rho$-saturating set in $\mathrm{PG}(N,q)$.
If the positions of a column of a parity check matrix of an $[n,n-r]_qR$ code are interpreted as homogeneous coordinates of a point in $\PG(r-1,q)$, then this matrix corresponds to an $(R-1)$-saturating $n$-set in $\PG(r-1,q)$, and vice versa. Thus, there is a \emph{one-to-one correspondence} between $[n,n-r]_qR$ codes and $(R-1)$-saturating $n$-sets in $\PG(r-1,q)$. This implies
\begin{equation*}
  \ell_q(r,R)=s_q(r-1,R-1).
\end{equation*}

    For an introduction to geometries over finite fields and their connections with coding theory, see \cite{libroBierbrauer,Dav95,DGMP-AMC,EtzStorm2016,Giul2013Survey,Hirs,HirsStor-2001,Klein-Stor,LandSt} and the references therein.

\subsection{The short content of this paper}\label{subsec1.4:content}
This paper is devoted to \emph{new constructive upper bounds on the binary length functions} $\ell_2(r,R)$.
If there is an $[n,n-r]_2R$ code, then $\ell_2(r,R)\le n$. Throughout the paper, let $\Delta(r,R)$ be the decrease of the known upper bounds on $\ell_2(r,R)$ provided by the new results. For $q=2$, $R=2,3,4$, we consider the known versions of the $q^m$-concatenating constructions of covering codes and propose new ones.  In addition, we obtain new useful partitions of column sets of parity check matrices for certain covering codes.  As a result, for $R=2,3,4$ we obtain new infinite families of covering codes that provide improved upper bounds on $\ell_2(r,R)$ and smaller asymptotic covering densities $\overline{\mu}(R)$ \eqref{eq11:liminfdens} compared with previously  known values; also, we obtain a few sporadic new covering codes, see Section~\ref{sec2:main res} for details.

The paper is organized as follows. In Section \ref{sec2:main res}, we collect the new results  and compare them with the known ones.
In Section \ref{sec3:qm concat}, we give a brief description of  a basic $q^m$-concatenating construction for $q=2$. In Sections \ref{sec4:R=2known}, \ref{sec6:R=3known}, \ref{sec8:R=4known}, we give the known results, useful for this paper. Finally, in Sections \ref{sec5:R=2new}, \ref{sec7:R=3new}, \ref{sec9:R=4new}, we obtain the new results.

\section{The main results}\label{sec2:main res}
Let the asymptotic covering density $\overline{\mu}(R)$ be as in \eqref{eq11:liminfdens}.
Let $t$ be an integer.

As far as the authors know, for $R=2,r=2t$, and $R=3,r=3t-1$, and $R=4,r=4t$, the best (from the point of view of the asymptotic density) known infinite families of binary covering $[n,n-r]_2R$ codes are obtained in \cite{DavCovNewConstrCovCod,DDL-ACCT3-1992,DDL-IEEE}, with the following parameters:
\begin{align*}
&R=2,\,r=2t\ge8,\,n=27\cdot2^{t-4}-1,\,\overline{\mu}(2)\thickapprox1.4238;\,\ell_2(2t,2)\le n\\
&\T{\cite[p.\ 53]{DDL-ACCT3-1992},\cite[Example 3.1]{DDL-IEEE}}.\db\\
&R=3,\,r=3t-1\ge41,\,n=821\cdot2^{t-9}-1,\,\overline{\mu}(3)\thickapprox1.3744;\,\ell_2(3t-1,3)\le n\\
&\T{\cite[p.\ 53]{DDL-ACCT3-1992},\cite[Example 4.2]{DDL-IEEE}}. \db\\
&R=4,\,r=4t\ge40,\,n=2944\cdot2^{t-10}-2,\,\overline{\mu}(4)\thickapprox2.8467;\ell_2(4t,4)\le n\T{\cite[Examples 6,\,9]{DavCovNewConstrCovCod}}.
\end{align*}

In this paper, these known results are improved. In Sections \ref{sec5:R=2new}, \ref{sec7:R=3new}, \ref{sec9:R=4new}, Theorems \ref{th5:r=18,28}, \ref{th5:R=2NewInfFam}, \ref{th7:r=26,38,41}, \ref{th7:newr=3t-1a}, \ref{th9:newFamR=4},
 we obtain new infinite families of binary covering codes, whose parameters are collected in Theorem \ref{th2:InfFamNew}.
Also, in Theorems \ref{th5:-3r=22,24,26}, \ref{th7:303R=3}, \ref{th7:r=26,38,41}, \ref{th9:r=31,n=690}, \ref{th9:r=48..64} we obtain  sporadic new covering codes, whose parameters are collected in  Theorem \ref{th2:SporasNew}. The new code families and sporadic new codes provide the decrease of the known upper bounds on $\ell_2(r,R)$, see the values $\Delta(r,R)$ in Theorems \ref{th2:InfFamNew}, \ref{th2:SporasNew}.
\begin{theorem}\label{th2:InfFamNew}
 There are new infinite families of covering $[n,n-r]_2R$ codes such that:
\begin{align}
&R=2,\,r=2t,\,r=10,18,20,\T{ and }r\ge28,\,n=26\cdot2^{t-4}-1,\,\overline{\mu}(2)\thickapprox1.3203; \label{eq2:newR2}\db\\
&\phantom{R}\ell_2(r,2)=\ell_2(2t,2)\le26\cdot2^{r/2-4}-1= 26\cdot2^{t-4}-1;\,\Delta(r,2)=2^{r/2-4}. \notag\db\\
&R=3,\,r=3t-1,\,r=26\T{ and }r\ge44,\,n=819\cdot2^{t-9}-1,\,\overline{\mu}(3)\thickapprox1.3643;\label{eq2:newR3} \db\\
&\phantom{R}\ell_2(r,3)=\ell_2(3t-1,3)\le819\cdot2^{(r-26)/3}-1=819\cdot2^{t-9}-1;
\,\Delta(r,3)=2^{(r-23)/3}.\notag\db\\
&R=4,\,r=4t,\, r=40\T{ and }r\ge68,\,n=2943\cdot2^{t-10}-1,\,\overline{\mu}(4)\thickapprox2.8428;
\label{eq2:newfam4t}\db\\
&\phantom{R}\ell_2(r,4)=\ell_2(4t,4)\le2943\cdot2^{r/4-10}-1=2943\cdot2^{t-10}-1,\,\Delta(r,4)=2^{r/4-10}-1.\notag
\end{align}
\end{theorem}

\begin{theorem}\label{th2:SporasNew}
There are new covering $[n,n-r]_2R$ codes such that:
\begin{align*}
&R=2,\,r=22,24,26,\,n=26.5\cdot2^{r/2-4}-3,\,\ell_2(r,2)\le n;\,\Delta(r,2)=2^{r/2-5}+2. \db\\
&R=3,\,r=38,41,\,n=820\cdot2^{(r-26)/3}-2,\,\ell_2(r,3)\le n;\,\Delta(r,3)=2^5+\lfloor r/41\rfloor.\db\\
&R=3,\,r=21,\,n=303,\,\ell_2(21,3)\le n;\,\Delta(21,3)=5.\db\\
&R=4,\,r=48,52,56,60,64,\,n=2944\cdot2^{r/4-10}-3,\,\ell_2(r,4)\le n,\,\Delta(r,4)=1.\db\\
&R=4,\,r=31,\,n=690,\,\ell_2(31,4)\le n;\,\Delta(31,4)=11.
\end{align*}
\end{theorem}

To obtain the new codes, in Sections \ref{sec5:R=2new}, \ref{sec7:R=3new}, \ref{sec9:R=4new}, we propose new versions of the $q^m$-concatenating constructions, see Theorems \ref{th5:-3}, \ref{th5:r=18,28}, \ref{th5:n0Psi}, \ref{th7:ell0=0R=3}, \ref{th9:ell0=1}.

\section{The $q^m$-concatenating constructions of covering\\ codes for $q=2$}\label{sec3:qm concat}

Throughout this paper, all matrices and columns are binary. An element of $\F_{2^m}$ written in a binary matrix denotes an $m$-dimensional binary column, that is a binary representation of this element; and vice versa, an $m$-dimensional binary column can be
viewed as an element of $\F_{2^m}$.

\subsection{Spherical $(R,\ell )$-capsules, $(R,\ell) $-objects, and partitions of the column set of a code parity check matrix}
\label{subsec31:capsul}
\begin{definition} \cite[Remark 5]{Dav90PIT}, \cite[p.\ 125]{DGMP-AMC}
Let $0\le\ell\le R$. A \emph{spherical }$(R,\ell)$-\emph{capsule} with center $\oc$
in $\F_2^{\,n}$ is the set $\{\ov:\ov\in \F_2^{\,n}$, $\ell\le d(\ov,\oc)\le R\}$.
\end{definition}

\begin{definition} \label{def31:partition} \cite[Definition 2.1]{DDL-IEEE}, \cite[Definition 2.1]{DGMP-AMC}
Let $\Hb\in\F_2^{\,r\times n}$ be a parity check $r\times n$ matrix
of an $ [n,n-r]_2R$ code $C$ and let $0\le\ell\le R$.
\begin{description}
  \item[(i)] A partition of the column set of the matrix
$\Hb$ into nonempty subsets is called an  $(R,\ell)$\emph{-partition} if every column of $F_2^{\,r}$ (including
the zero column) is equal to a sum of at least $\ell $ and at most $R$ columns of
$\Hb$ belonging to distinct subsets.  For an $(R,0)$-partition we can formally treat the zero column as the sum of 0 columns. We use the term ``$R$-\emph{partition}'' when the value of $\ell$ is not relevant or not known.

  \item[(ii)] If $\Hb$ admits an $(R,\ell)$-partition, the code $C$ is called an $(R,\ell) $\emph{-object} of the space $\F_2^{\,n}$ and is denoted as an $[n,n-r]_2R,\ell $ or an $[n,n-r,d]_2R,\ell$ code. It is not necessary that $\ell $ is the greatest value with the properties considered. Therefore, any $(R,\ell )$-partition with $\ell >0$ is also an $(R,\ell_1)$-partition with $\ell_1=0,1,\ldots,\ell -1$.  So, an $[n,n-r]_2R,\ell$ code with $\ell > 0$ is also an $[n,n-r]_2R,\ell_1$ code with $\ell_1 = 0, 1,\ldots,\ell - 1$.
\end{description}
\end{definition}

Let $p(\Hb,\ell;\Ps)$ be the number of subsets in an $(R,\ell)$-partition $\Ps$ of a parity check matrix $\Hb$ of an $[n,n-r]_2R,\ell$ code. Obviously, $R \le p(\Hb,\ell; \Ps)\le n$. If $p(\Hb,\ell;\Ps)=n$, i.e. every subset contains one column, then the partition $\Ps$ is called \emph{trivial}. The trivial partition of a parity check matrix of an $ [n,n-r]_2R,\ell $ code is an $(R,\ell )$-partition. We use the notation $p(\Hb,\ell)$ when the partition $\Ps$ is not relevant. In other words, the notation $p(\Hb,\ell)$ means that the matrix  $\Hb$ admits an $(R,\ell)$-partition into $p(\Hb,\ell)$ subsets. For the notations $p(\Hb,\ell;\Ps)$ and $p(\Hb,\ell)$, the value of $R$  is defined by the context.

\begin{lemma}\label{Lem31:d&elll}
\emph{\cite{Dav90PIT,Dav95,DDL-IEEE}, \cite[Lemma 2.2]{DGMP-AMC}} An $[n,n-r,d]_2R$ code
is an $[n,n-r,d]_2R,\ell $ code with $\ell \geq 1$ if and only if $d\leq
R. $ If $d>R$ the maximum possible value of $\ell $ is zero$.$
\end{lemma}

Throughout the paper, all examined codes have minimum distance $d\ge3$. So, all codes of covering radius $R=2$ are $(2,0)$-objects.
Therefore, for codes of covering radius $R=2$, for short, we use the notations $[n,n-r]_22$ and $[n,n-r,d]_22$ codes, 2-objects, 2-partition, $p(\Hb;\Ps)$, $p(\Hb)$.

\begin{lemma} \emph{\cite[p.\ 125]{DGMP-AMC}}
 Spherical $(R,\ell)$-capsules centered at vectors of an $(R,\ell)$-object cover the space $\F_2^{\,n}.$
\end{lemma}

\subsection{Basic $q^m$-concatenating Construction QM for $\boldsymbol{q=2}$}\label{subsec32:QM}
The $q^m$-concatenating constructions were proposed in 1990 in \cite{Dav90PIT} and then investigated and developed in numerous works, see e.g. \cite{Handbook-coverings,Dav95,DavParis,DDL-ACCT2-1990,DDL-ACCT3-1992,DDL-IEEE,DGMP-AMC,Giul2013Survey}, \cite[Section 5.4]{CHLL-bookCovCod}, \cite[Supplement]{Struik}, and the references therein. The term ``$q^m$-concatenating constructions'' is  not always used in these papers.

Starting from an $[n_0,n_0-r_0]_qR$ code of length $n_0$, the $q^m$-concatenating constructions yield  new $[n,n-(r_0+Rm)]_qR$ codes with the same covering radius $R$ and length $ n=q^m n_0+N_m$, where $m$ must satisfy some conditions and $N_m\le R\theta _{m,q}$, with $\theta_{m,q}=(q^m-1)/(q-1)$, $\theta_{m,2}=2^m-1$.
The covering density~\eqref{eq11:CovDensity} of the new codes is almost the same (slightly greater) as for the starting code. Using the constructions iteratively, we can obtain infinite families of new $ [n,n-(r_0+Rm)]_qR$ codes where $m$ ranges over an infinite set of integers.

\textbf{Basic Construction QM.}
We use the results and ideas of \cite{Dav90PIT,Dav95,DDL-ACCT3-1992,DDL-IEEE,DGMP-AMC}.

Let $\Hb_0$ be a parity check $r_0\times n_0$ matrix of an $[n_0,n_0-r_0]_2R,\ell_0$ \emph{starting} code $C_0$,
\begin{equation}\label{eq32:H0}
  \Hb_0=[\hb_1\hb_2\ldots\hb_{n_0}],~\hb_j\in\F_2^{\,r_0},~j=1,\ldots,n_0,
\end{equation}
where $\hb_j$ is a binary $r_0$-positional column. Assume that $\Hb_0$ admits a starting $(R,\ell _0)$-partition $\Ps_0$ into $p(\Hb_0,\ell _0)$ subsets. Let $m\geq 1$ be an integer depending on $\Ps_0$ and $n_0$. With every column $\hb_j$ we associate an element $\beta_j$ $\in \F_{2^m}\cup \{\ast\}$ so that $\beta_i\neq \beta_j$
if columns $\hb_i$ and $\hb_j$ belong to \emph{distinct }subsets of $ \Ps_0$. If $\hb_i$ and $\hb_j$ belong to the same subset we are free to assign either $\beta_i=\beta_j$ or $\beta_i\ne \beta_j$. We call $\beta_j$ an \emph{indicator} of the column $\hb_j$. Let $\Bs=\{\beta_1,\beta _2,\ldots,\beta_{n_0}\}$ be an \emph{indicator set}. The condition $\#\Bs \geq p(\Hb_0,\ell _0)$ is necessary.

Let $\D$ be an $(r_0+Rm)\times N_m$ matrix with \emph{the $r_0$ top rows equal to the zero vector} and with $N_m\le (R-\ell_0)\theta_{m,2}$. For $j=1,2,\dots,n_0$, we denote by  $\Ab(\hb_j,\beta_j)$ an $(r_0+Rm)\times 2^m$ matrix, in which the \emph{$r_0$ top rows are formed by $2^m$-fold repeating of the column $\hb_j$}. Let $\mathbf{0}_{v}$ be the zero matrix with $v$ rows where the number of columns (possibly, 1) is clear by the context.

Finally, define a new code $C$ as an $[n,n-(r_0+Rm)]_2R_C$ code with $n=2^m n_0+N_m$ and the parity check $(r_0+Rm)\times n$ matrix $\Hb_C$ of the following form:
\begin{align}
& \Hb_C \triangleq\left[ \D~\Ab(\hb_1,\beta_1)~\Ab(\hb_2,\beta_2)~\ldots ~
\Ab(\hb_{n_0},\beta_{n_0})\right],\label{eq32:QM-H}\db\\
&\Ab(\hb_j,\beta_j)\triangleq\left[ \renewcommand{\arraystretch}{1.1}
\begin{array}{cccc}
\hb_j & \hb_j & \mathbf{\cdots } & \hb_j \\
\xi _1 & \xi _2 & \cdots & \xi _{q^m} \\
\beta _j\xi _1 & \beta _j\xi _2 & \cdots & \beta _j\xi _{q^m} \\
\beta _j^{2}\xi _1 & \beta_j^{2}\xi _2 & \cdots & \beta _j^{2}\xi
_{q^m} \\
\vdots & \vdots & \vdots & \vdots \\
\beta _j^{R-1}\xi _1 & \beta_j^{R-1}\xi _2 & \cdots & \beta
_j^{R-1}\xi _{q^m}
\end{array}
\right] \T{if }\beta _j\in \F_{2^m},\label{eq32:Aj}\db\\
&\Ab(\hb_j,\beta_j)\triangleq\left[
\begin{array}{cccc}
\hb_j & \hb_j & \mathbf{\cdots } & \hb_j \\
\zb_{(R-1)m} & \zb_{(R-1)m} & \cdots & \zb_{(R-1)m} \\
\xi _1 & \xi _2 & \cdots & \xi _{q^m}
\end{array}
\right] \T{ if }\beta _j=\ast ,  \label{eq32:A*}\db\\
&\phantom{\Ab(\hb_j,\beta_j)\triangleq~}\{\xi _1,\xi _2,\ldots ,\xi _{q^m}\}=\F_{2^m},~\xi_1=0. \label{eq32:xi}
\end{align}

If $m$, $\D$, and $\Bs$ are carefully chosen, then the covering radius $R_C$ of the new code $C$ is equal to the
covering radius $R$ of the starting code $C_0$; see  Sections \ref{sec4:R=2known}--\ref{sec9:R=4new}, where we use
the following notations:
\begin{align}
  & \Wb_m\T{ is the parity check $m\times(2^m-1)$ matrix of the }[2^m-1,2^m-1-m]_21\label{eq32:Wbm}\db\\
  &\T{Hamming code}; \Wb_m\setminus\wb \T{ is the matrix $\Wb_m$ without a column }\wb.\notag\db\\
&\D_1(R)=\left[\begin{array}{@{}c@{}}
 \zb_{r_0+(R-1)m}\\
 \Wb_m
 \end{array}\right],\D_2(R)=\left[\begin{array}{cc}
 \zb_{r_0+(R-2)m}&\zb_{r_0+(R-2)m}\\
 \Wb_m&\zb_{m}\\
 \zb_{m}&\Wb_m
 \end{array}\right],\D_3=\left[\begin{array}{c}
 \zb_{r_0+m}\\
 \Wb_m\\
 \zb_{m}
 \end{array}\right].\label{eq32:varD}\db\\
 &\D_4=\left[\begin{array}{@{}cc@{}}
 \zb_{r_0}&\zb_{r_0}\\
 \Wb_m&\zb_{m}\\
 \zb_{2m}&\Hcb_{2m}
 \end{array}\right],
 \D_5=\left[\begin{array}{@{}cc@{}}
 \zb_{r_0+m}&\zb_{r_0+m}\\
 \Hcb_{2m}&\zb_{2m}\\
\zb_{m}&\Wb_m
 \end{array}\right],
 \D_6=\left[\begin{array}{@{}ccc@{}}
 \zb_{r_0}&\zb_{r_0}&\zb_{r_0}\\
 \Wb_m\setminus\wb&\wb&\zb_m\\
 \zb_m&\wb&\Wb_m\setminus\wb
  \end{array}\right].\label{eq32:varD2}\db\\
  &\Hcb_{2m}\T{ is a parity check }2m\times n_{2m} \T{ matrix of an }[n_{2m},n_{2m}-2m]_22\T{ code }
\Vc_{2m}.\label{eq32:Hcb2m}\db\\
  &\T{QM}_j^R\T{ is the $j$-th version of Construction QM for covering radius }R.\notag\end{align}

\section{The known results on binary linear codes of covering radius 2}\label{sec4:R=2known}

\begin{theorem}\label{th4:DDLR=2} \emph{\cite[Theorem 3.1]{DDL-IEEE}}
$\mathbf{Constructions~QM_1^2,~QM_2^2}$. In Construction \emph{QM} of Section~$\ref{subsec32:QM}$, let the starting code $C_0$  be an $[n_0,n_0-r_0]_22$ code with a parity check matrix $\Hb_0$ as in \eqref{eq32:H0}. For $r=r_0+2m$, $n=2^mn_0+N_m$, we define a new $[n,n-r]_2R_C$ code $C$ by a parity check
$r\times n$ matrix $\Hb_C$ of the form \eqref{eq32:QM-H}--\eqref{eq32:varD}. Let the indicator set $\Bs$, parameter $m$, and $r\times N_m$ submatrix $\D$ be chosen by one of the ways \eqref{eq4:CaseA1inp}, \eqref{eq4:CaseA2inp}, which set specific  Constructions
\emph{QM}$_1^2$, \emph{QM}$_2^2$. We have
\begin{align}
&\emph{QM}_1^2:~\Bs\subseteq\F_{2^m}\cup\{*\},~2^m+1\ge p(\Hb_0),~ \D=\D_2(2).\label{eq4:CaseA1inp}\db\\
&\emph{QM}_2^2:~\Bs=\F_{2^m},~n_0\ge2^m\ge p(\Hb_0),~
 \D=\D_1(2). \label{eq4:CaseA2inp}
\end{align}
Then the new code $C$ is an $[n, n - r]_22$ code with the parameters:
\begin{align}
&\emph{QM}_1^2:~R=2,~n=2^m(n_0+2)-2,~r = r_0 + 2m,~p(\Hb_C)\le p(\Hb_0).\label{eq4:CaseA1res}\db\\
&\emph{QM}_2^2:~R=2,~n=2^m(n_0+1)-1,~r = r_0 + 2m,~p(\Hb_C)\le 2^{m+1}+1.\label{eq4:CaseA2res}
\end{align}
\end{theorem}

We introduce the notation.
\begin{align}
&\phi(r)=\phi(2t)\triangleq27\cdot2^{r/2-4}-1 = 27\cdot2^{t-4}-1  \T{ for } r=2t\label{eq4:notatKnwR=2even};\db\\
&f(r)=f(2t -1)\triangleq5\cdot2^{(r-3)/2}-1= 5\cdot2^{t-2}-1   \T{ for }  r=2t -1.\label{eq4:notatKnwR=2odd}
\end{align}

\begin{theorem}\label{th4:KnownFamR=2}
Let the asymptotic covering density $\overline{\mu}(R)$ be as in \eqref{eq11:liminfdens}. Let $\phi(r)$ and $f(r)$ be as in \eqref{eq4:notatKnwR=2even} and \eqref{eq4:notatKnwR=2odd}. There exist the following infinite families of binary linear $[n,n-r]_22$ covering codes with  growing codimension $r$ and parameters as in \eqref{eq4:KnwFamR=2even r}, \eqref{eq4:KnwFamR=2odd r}.
\begin{description}
  \item[(i)] \emph{\cite[p.\ 53]{DDL-ACCT3-1992}, \cite[Example 3.1, equation (1.3)]{DDL-IEEE}, \cite[Theorem 5.4.27(i), equation (5.4.30)]{CHLL-bookCovCod}}
\begin{align}
&R=2,~r=2t\ge8,~t\ge4,~n=27\cdot2^{t-4}-1=\phi(2t),~\overline{\mu}(2)\thickapprox1.4238;
\label{eq4:KnwFamR=2even r}\db\\
&\ell_2(2t,2)\le 27\cdot2^{t-4}-1=\phi(2t),~t\ge4.\notag
\end{align}
  \item[(ii)] \emph{\cite[Theorem 1, equation (5)]{GabDavTombR=2}, \cite[Remark 3.1]{DDL-IEEE}, \cite[Theorem 5.4.27(ii)]{CHLL-bookCovCod}}
\begin{align}
&R=2,~r=2t-1\ge3,~t\ge2,~n=5\cdot2^{t-2}-1=f(r),~\overline{\mu}(2)\thickapprox1.5625;\label{eq4:KnwFamR=2odd r}\db\\
&\ell_2(2t-1,2)\le 5\cdot2^{t-2}-1=f(2t-1),~t\ge2.\notag
\end{align}
\end{description}
\end{theorem}

\begin{theorem}\label{th4:KR code}\emph{Kaikkonen, Rosendahl\cite[p.\ 1812]{KaikRoseADSlike2003}}
Let $C_{KR}$ be the binary $[51,41]_2$ code with the parity check matrix
$\Hb_{KR}=[\Ib_{10}\Mb_{KR}]$, where $\Mb_{KR}$ is the $10\times 41$ binary matrix with the following columns in
hexadecimal notation:
\begin{align}\label{eq4:HKR}
&\Mb_{KR}=[\mathrm{1B6,193,1CC,187,1F6,F7,16E,140,3C,296,22F,303,381,365,}\db\\
&\mathrm{11D,1A3,274,2F2,254,56,F,41,357,208,34,329,28D,31D,3D5,129,3D7,}\notag\db\\
&\mathrm{B7,3EC,2E2,23C,AD,34E,155,2E6,371,D4}].\notag
\end{align}
Then $C_{KR}$ is a $[51,41]_22$ code with covering radius $2$.
\end{theorem}

\section{New results on binary linear codes of covering radius 2}\label{sec5:R=2new}

\begin{theorem}\label{th5:-3}
$\mathbf{Construction~QM_3^2}$. In Construction \emph{QM} of Section $\ref{subsec32:QM}$, let the starting code $C_0$  be an $[n_0,n_0-r_0,3]_22$ code with a parity check matrix $\Hb_0$ of the form \eqref{eq32:H0}. Let $\Ps_0$ be a $2$-partition
of $\Hb_0$ into $p(\Hb_0)$ subsets. We define a new $[n,n-r,d_C]_2R_C$ code $C$ by a parity check
matrix $\Hb_C$ of the form \eqref{eq32:QM-H}--\eqref{eq32:varD2}, where the indicator set $\Bs$, parameter $m$, and the submatrix $\D$  are as follows
\begin{align}\label{eq5:-3inp}
&\emph{QM}_3^2:~\Bs\subseteq\{*\}\cup\F_{2^m}\setminus\{1\},~2^m\ge p(\Hb_0),~\D=\D_6.
\end{align}
Then the new code $C$ is a  binary linear $[n, n - r,3]_22$ code with the parameters:
\begin{equation}\label{eq5:-3res}
\emph{QM}_3^2:~R=2,~n=2^m(n_0+2)-3,~r = r_0 + 2m.
\end{equation}
\end{theorem}

\begin{proof}
The values of $n,r$ and $d_C=3$ directly follow from Construction QM and \eqref{eq5:-3inp}.

We prove $R_C=2$, considering the possible representation as a sum of columns of $\Hb_C$ for an arbitrary column $\Ub=(\pib,\ub_1,\ub_2)^{tr}\in\F_2^{\,r}$, where $\pib\in\F_2^{\,r_0}$, $\ub_1,\ub_2\in\F_2^{\,m}$.
\begin{description}
  \item[(1)] $\pib=\hb_i+\hb_j$; $\hb_i$ and $\hb_j$ belong to distinct subsets of $\Ps_0$; $\beta_i\ne\beta_j$.

  If $\beta_i,\beta_j\ne *$, we find $\xb,\yb$ from the system $\xb+\yb=\ub_1$, $\beta_i \xb+\beta_j \yb=\ub_2$.
  Now $\Ub= (\hb_i,\xb,\beta_i \xb)^{tr} +(\hb_j,\yb,\beta_j \yb)^{tr}$. If $\beta_j=*$, $\Ub= (\hb_i,\ub_1,\beta_i\ub_1)^{tr} +(\hb_j, \zb_m,\beta_i\ub_1+\ub_2)^{tr}$.

  \item[(2)] $\pib=\hb_j$.

  Let $\beta_j\ne *$, $\beta_j\ub_1+\ub_2\ne \wb$. Then $\Ub=(\hb_j, \ub_1, \beta_j\ub_1)^{tr} +(\zb_{r_0+m}, \beta_j\ub_1+\ub_2)^{tr}$.

  Let $\beta_j\notin \{*,0$\};  $\beta_j\ub_1+\ub_2= \wb$. Then $\beta_j^{-1}\ub_2+\ub_1=\beta_j^{-1}\wb\ne\wb$, as $\beta_j\ne1$ by hypothesis. We have
   $\Ub=(\hb_j, \beta_j^{-1}\ub_2,\ub_2)^{tr}+(\zb_{r_0},\beta_j^{-1}\ub_2+\ub_1,\zb_m)^{tr}$.

   Let $\beta_j=0$, $\beta_j\ub_1+\ub_2= \wb$. Then $\ub_2=\wb$, $\Ub=(\hb_j,\ub_1+\wb,\zb_m)^{tr} +
  (\zb_{r_0}, \wb,\wb)^{tr}$.

   Let $\beta_j=*$, $\ub_1\ne\wb$. Then $\Ub=(\hb_j,\zb_m,\ub_2)^{tr}+(\zb_{r_0},\ub_1,\zb_m)^{tr}$.

   Let $\beta_j=*$, $\ub_1=\wb$. Then $\Ub=(\hb_j,\zb_m,\ub_2+\wb)^{tr}+(\zb_{r_0},\wb,\wb)^{tr}$.

  \item[(3)] $\pib=\zb_{r_0}$.

  The bottom $2m$ rows of $\D$ form a parity check matrix of the amalgamated direct sum (ADS), see  \cite[Section 4.1]{CHLL-bookCovCod}, \cite{GrahSlo}, of the two $[2^m-1,2^m-1-m,3]_21$ Hamming codes that gives a code of covering radius 2.
\end{description}

So, any column $\Ub \in\F_2^{\,r}$  can be represented as a column of $\Hb_C$   or a sum of two columns of $\Hb_C$. By Definition~\ref{Def11:CoverRad}(ii), the covering radius of the new code $C$ is $R_C=2$.
\end{proof}

\begin{theorem}\label{th5:PartitionKR}
Let the $[51,41]_22$ code $C_{KR}$ \emph{\cite{KaikRoseADSlike2003}} be given by the parity check matrix $\Hb_{KR}$, see Theorem \emph{\ref{th4:KR code}}. Let the columns of $\Hb_{KR}$ be numbered from left to right, i.e. $\Hb_{KR}=[\hb_1, \hb_2,\ldots,\hb_{51}]$, where $\hb_i\in\F_2^{10}$, $\hb_1=(10~0000~0000)^{tr}$, $\hb_2=(01~0000~0000)^{tr}$,\ldots,$\hb_{10}=(00~0000~0001)^{tr}$, $\hb_{11}=1B6=(01~1011~0110)^{tr}$,\ldots,$\hb_{51}=D4=(00~1101~0100)^{tr}$.
\begin{description}
  \item[(i)]
  Let the partition $\Ps_{KR}$ of the column set of $\Hb_{KR}$ into $11$ subsets be as follows:
  \begin{align}\label{eq5:PKR}
&   \Ps_{KR}\triangleq\{5,13,43\},\{20,27\},\{3,29,33,39,41,48,51\},\{1,7,19,25,34,45\},\db\\
&\{2,4,18\},\{6,8,12,26,28,35,44\},\{9,22,23,30\},\{10,11,15,16,32,42\},\notag\db\\
&\{14,24,49,50\},\{17,21,31,37,46,47\},\{36,38,40\},\notag
  \end{align}
where column numbers for every subset are written.

Then $\Ps_{KR}$ is a $2$-partition of $\Hb_{KR}$, i.e.  $p(\Hb_{KR};\Ps_{KR})=11$.

  \item[(ii)] The code $[51,41]_22$ $C_{KR}$ has minimum distance $d=3$. The columns $\hb_5$, $\hb_{27}$, $\hb_{29}$ of $\Hb_{KR}$, the sum of which is the zero column, belong to distinct subsets of $\Ps_{KR}$.
\end{description}
\end{theorem}

\begin{proof}
  We obtained the 2-partition $\Ps_{KR}$ by computer search.
Also, by \eqref{eq4:HKR}, $\hb_5+\hb_{27}+\hb_{29}=(00~ 0010~0000)^{tr}+(10~ 0111~ 0100)^{tr}+(10~0101~ 0100)^{tr}=\zb_{10}$.
\end{proof}

We introduce the notations.
\begin{align}
&\Phi(r)=\Phi(2t)\triangleq26\cdot2^{r/2-4}-1= 26\cdot2^{t-4}-1 \T{ for }r=10,18,20,~r=2t\ge28.\label{eq5:notEvenR=2New1}\db\\
&\widehat{\Phi}(r)\triangleq26.5\cdot2^{r/2-4}-3= 26.5\cdot2^{t-5}-3 \T{ for }r=22,24,26.\label{eq5:notEvenR=2New2}
\end{align}

\begin{lemma}\label{lem5:Delta}
 Let $r=2t$. We have
 \begin{equation}\label{eq5:Delta}
  \phi(r)-\Phi(r)=2^{r/2-4}=2^{t-4}; ~\phi(r)-\widehat{\Phi}(r)=2^{r/2-5}+2=2^{t-5}+2.
 \end{equation}
\end{lemma}

\begin{proof}
  The assertions follow directly from \eqref{eq4:notatKnwR=2even}, \eqref{eq5:notEvenR=2New1}, \eqref{eq5:notEvenR=2New2}.
\end{proof}

\begin{theorem}\label{th5:r=18,28}
$\mathbf{Constructions~QM_4^2,~QM_5^2}$.
Let the covering density $\mu$ be as in \eqref{eq11:CovDensity}, $\Phi(r)$ be as in \eqref{eq5:notEvenR=2New1}.
There exist two new $[n,n-r,3]_22$ codes $C$ with a parity check matrix $\Hb_C$ of the form \eqref{eq32:QM-H}--\eqref{eq32:varD} and parameters as in \eqref{eq5:r=18}, \eqref{eq5:r=28}.
\begin{align}
&\emph{\textbf{(i)}}~\emph{QM}_4^2:~R=2,~ r=18,~n=831=\Phi(18),~p(\Hb_C)\le2^5+1, ~\mu\thickapprox1.31873, \label{eq5:r=18}\db\\
&\ell_2(18,2)\le831=\Phi(18),~\Delta(18,2)=2^5.\notag\db\\
&\emph{\textbf{(ii)}}~\emph{QM}_5^2:~R=2,~r=28,~n=26623=\Phi(28),~p(\Hb_C)\le2^6+2,~ \mu\thickapprox1.32026, \label{eq5:r=28}\db\\
&\ell_2(28,2)\le26623=\Phi(28),~\Delta(28,2)=2^{10}\notag.
\end{align}
\end{theorem}

\begin{proof}
  \textbf{(i)} \textbf{Construction QM}$\mathbf{_4^2}$.  In Construction QM of Section $\ref{subsec32:QM}$, let the starting code $C_0$  be the $[51,41,3]_22$ code $C_{KR}$ with the parity check matrix $\Hb_{KR}=\Hb_0$, see Theorems \ref{th4:KR code}, \ref{th5:PartitionKR}. So, $n_0=51$, $r_0=10$. We take $m=4$ and obtain a 2-partition $\Ps_0=\Ps_{KR}^*$ of $16=2^m$ subsets, partitioning the first three
subsets of $\Ps_{KR}$ \eqref{eq5:PKR}:
\begin{align}\label{eq5:PKR*}
&   \Ps_{KR}^*\triangleq\{5\},\{27\},\{29\},\{13\},\{43\},\{20\},\{3\},\{33,39,41,48,51\},\\
&\{1,7,19,25,34,45\},\{2,4,18\},\{6,8,12,26,28,35,44\},\{9,22,23,30\},\notag\\
&\{10,11,15,16,32,42\},\{14,24,49,50\},\{17,21,31,37,46,47\},\{36,38,40\}.\notag
  \end{align}
 We put the indicator set $\Bs=\F_{2^m}$  and assign to all columns of every subset of $\Ps_{KR}^*$ the same indicator different from the indicators of the other subsets.
 Finally, we take $\D=\D_1(2)$ and obtain the matrix $\Hb_C$ of the form \eqref{eq32:QM-H}--\eqref{eq32:varD}, which defines the new $[n,n-r,d_C]_2R_C$ code $C$ with $n=2^4\cdot51+2^4-1=831=\Phi(18)$, $r=10+2\cdot4=18$.

 The minimum distance $d_C=3$ follows from the submatrix $\D$.

  To prove $R_C=2$ we consider the possible representations as a sum of columns of $\Hb_C$ for an arbitrary column $\Ub=(\pib,\ub_1,\ub_2)^{tr}\in\F_2^{\,r}$, where $\pib\in\F_2^{\,r_0},\ub_1,\ub_2\in\F_2^{\,m},r_o=10,m=4$.
\begin{description}
  \item[(1)] $\pib=\hb_i+\hb_j$; $\hb_i$ and $\hb_j$ belong to distinct subsets of $\Ps_0$; $\beta_i\ne\beta_j$.

   We find $\xb$ and $\yb$ solving the linear system $\xb+ \yb=\ub_1$, $\beta_i \xb+\beta_j \yb=\ub_2$.
  Now $\Ub= (\hb_i,\xb,\beta_i \xb)^{tr} +(\hb_j,\yb,\beta_j \yb)^{tr}$.

  \item[(2)] $\pib=\hb_j$.

  $\Ub=(\hb_j, \ub_1, \beta_j\ub_1)^{tr} +(\zb_{r_0+m}, \beta_j\ub_1+\ub_2)^{tr}$, where $(\zb_{r_0+m}, \beta_j\ub_1+\ub_2)^{tr}\in\D$.

  \item[(3)] $\pib=\zb_{r_0}$.

  Let $\ub_1\ne0$. As $\Bs=\F_{2^m}$, indicator $\beta_k = \ub_2/\ub_1$
 exists; $\Ub=(\hb_k, \zb_m, \zb_m)^{tr} +(\hb_k, \ub_1, \beta_k\ub_1)^{tr}$. If $\ub_1=0$, then $\Ub=(\zb_{r_0+m}, \ub_2)^{tr}\in\D$.
\end{description}
 So, any column $\Ub \in\F_2^{\,r}$  can be represented as a column of $\Hb_C$   or a sum of two columns of $\Hb_C$. By Definition~\ref{Def11:CoverRad}(ii), the covering radius of the new code $C$ is $R_C=2$ that implies new upper bound on the length function $\ell_2(18,2)\le 831=\Phi(18)$. By \eqref{eq4:KnwFamR=2even r} and \eqref{eq5:Delta},  the decrease $\Delta(18,2)=2^{18/2-4}=2^5$.

Now, we construct the $2$-partition $\Ps_{C}$ of $2^{m+1}+1=2^5+1$ subsets.

We obtain $2^{m}=16$ interim subsets, partitioning the columns of $\Hb_C$ (except $\D$) so that if columns $\hb_i, \hb_j$ of $\Hb_0=\Hb_{KR}$ belong to distinct subsets of
$\Ps_{KR}^*$ \eqref{eq5:PKR*}, then the columns of submatrices $\Ab(\hb_i,\beta_i)$, $\Ab(\hb_j,\beta_j)$ belong to distinct subsets
of the partition of $\Hb_C$.
Then we obtain $2^{m+1}=2^5$ subsets of $\Ps_{C}$ partitioning every interim subset into two
subsets such that the first one contains all columns of the
form $(\hb_\bullet, \zb_4, \zb_4)^{tr}$, where $\xi_1=0$ is used, and the second one contains the rest of the
columns. This is connected with the case  $\pib=\zb_{r_0}$, $\ub_1\ne0$. The  $(2^5+1)$-th subset consists of the columns of $\D$.

By above, including Theorem \ref{th5:PartitionKR}(ii) and \eqref{eq5:PKR*}, there are subsets of $\Ps_{C}$ consisting of one column such that
\begin{align}\label{eq5:3subsets}
   &\{(\hb_5,\zb_8)^{tr}\},~\{(\hb_{27},\zb_8)^{tr}\},~\{(\hb_{29},\zb_8)^{tr}\};~
   (\hb_5,\zb_8)^{tr}=(\hb_{27},\zb_8)^{tr}+(\hb_{29},\zb_8)^{tr}.
\end{align}

 \textbf{(ii)}  \textbf{Construction QM}$\mathbf{_5^2}$.  In Construction QM of Section $\ref{subsec32:QM}$,
 let the starting code $C_0$  be the $[831,813,3]_22$ code obtained in the case (i). So, $n_0=831$, $r_0=18$. The parity check matrix $\Hb_C$ and the 2-partition $\Ps_C$ from the case~(i) are the parity check $18\times 831$ matrix $\Hb_0$ and the 2-partition $\Ps_0$ of $2^5+1$ subsets (including ones from \eqref{eq5:3subsets}) for the case (ii).  For the case (ii), we have, see \eqref{eq5:3subsets}:
 \begin{align}\label{eq5:Hb_0}
 &\Hb_0 =[\hb_1\hb_2\hb_3\ldots\hb_{831}],~\hb_i\in\F_2^{~18},~\hb_1=(00~0010~0000~0000~0000)^{tr},\db\\
  & \hb_2=(10~0111~0100~0000~0000)^{tr},~\hb_3=(10~0101~0100~0000~0000)^{tr};\notag\db\\
  & \hb_1=\hb_2+\hb_3; ~\Ps_0 =\{\hb_1\},\{\hb_2\},\{\hb_3\},\ldots,\{\T{columns of }\D\}.\notag
 \end{align}

We take $m=5$. We put the indicator set $\Bs=\F_{2^m}\cup\{*\}$  and assign to all columns of every subset of $\Ps_0$ the same indicator other than the indicators of other subsets; in that, we put $\beta_1=*$. Finally, we take $\D=\D_1(2)$ and obtain the matrix $\Hb_C$ of the form \eqref{eq32:QM-H}--\eqref{eq32:varD}, which defines the new $[n,n-r,d_C]_2R_C$ code $C$ with $n=2^5\cdot831+2^5-1=26623=\Phi(28)$, $r=18+2\cdot5=28$.

 The minimum distance $d_C=3$ follows from $\D$.
  To prove $R_C=2$ we consider an arbitrary column $\Ub=(\pib,\ub_1,\ub_2)^{tr}\in\F_2^{\,r}$, $\pib\in\F_2^{\,r_0},\ub_1,\ub_2\in\F_2^{\,m},r_o=18,m=5$.
\begin{description}
  \item[($1'$)] $\pib=\hb_i+\hb_j$; $\hb_i$ and $\hb_j$ belong to distinct subsets of $\Ps_0$; $\beta_i\ne\beta_j$.

    Let $\beta_i,\beta_j\ne *$. We find $\xb$ and $\yb$ from the linear system $\xb+ \yb=\ub_1$, $\beta_i \xb+\beta_j \yb=\ub_2$.
  Now $\Ub= (\hb_i,\xb,\beta_i \xb)^{tr} +(\hb_j,\yb,\beta_j \yb)^{tr}$.

  Let $j=1$ that implies $\beta_j=*$. Then $\Ub= (\hb_i,\ub_1,\beta_i\ub_1)^{tr} +(\hb_1, \zb_m,\beta_i\ub_1+\ub_2)^{tr}$.

  \item[($2'$)] $\pib=\hb_j$.

  If $j\ne1$, $\beta_j\ne *$, then $\Ub=(\hb_j, \ub_1, \beta_j\ub_1)^{tr} +(\zb_{r_0+m}, \beta_j\ub_1+\ub_2)^{tr}$ with $(\zb_{r_0+m}, \beta_j\ub_1+\ub_2)^{tr}\in\D$. If $j=1$, then, by \eqref{eq5:Hb_0}, $\pib=\hb_2+\hb_3$ and we return to the case ($1'$).

  \item[($3'$)] $\pib=\zb_{r_0}$.

  Let $\ub_1\ne0$. As $\Bs=\F_{2^m}\cup\{*\}$, indicator $\beta_k = \ub_2/\ub_1\ne*$
 exists. We have $\Ub=(\hb_k, \zb_m, \zb_m)^{tr} +(\hb_k, \ub_1, \beta_k\ub_1)^{tr}$.
If $\ub_1=0$, then $\Ub=(\zb_{r_0+m}, \zb_m,\ub_2)^{tr}\in\D$.
\end{description}
So, any column $\Ub \in\F_2^{\,r}$  can be represented as a column of $\Hb_C$ or a sum of two columns of $\Hb_C$. By Definition~\ref{Def11:CoverRad}(ii), the covering radius of the new code $C$ is $R_C=2$, that implies new upper bound on the length function $\ell_2(28,2)\le26623=\Phi(28)$. By \eqref{eq4:KnwFamR=2even r} and \eqref{eq5:Delta}, the decrease $\Delta(28,2)=2^{28/2-4}=2^{10}$.

Now, we construct the $2$-partition $\Ps_C$ of $2^{m+1}+2=2^6+2$ subsets.
We obtain $2^{m}+1=33$ interim subsets, partitioning the columns of $\Hb_C$ (except $\D$) so that if columns $\hb_i, \hb_j$ of $\Hb_0$ belong to distinct subsets of
$\Ps_0$ \eqref{eq5:Hb_0}, then the columns of submatrices $\Ab(\hb_i,\beta_i)$, $\Ab(\hb_j,\beta_j)$ belong to distinct subsets
of the partition of $\Hb_C$.
Then we obtain $2^{m+1}=2^6$ subsets of $\Ps_{C}$ partitioning every interim subset (except $\{\Ab(\hb_1,*)\}$) into two
subsets such that the first one contains all columns of the
form $(\hb_\bullet, \zb_5, \zb_5)^{tr}$, where $\xi_1=0$ is used, and the second one contains the rest of the
columns. This is connected with the case  $\pib=\zb_{r_0}$, $\ub_1\ne0$, in which the columns of $\{\Ab(\hb_1,*)\}$ do not appear; these columns form $(2^6+1)$-th subset. The  $(2^6+2)$-th subset consists of the columns of~$\D$.
\end{proof}

\begin{theorem}\label{th5:-3r=22,24,26}
Let $\widehat{\Phi}(r)$ be as in \eqref{eq5:notEvenR=2New2}.
There are $3$ new  $[n,n-r,3]_22$ codes $C$:
\begin{equation}\label{eq5:r=22,24,26}
R=2,~r=22,24,26,~n=\widehat{\Phi}(r),~\ell_2(r,2)\le\widehat{\Phi}(r),~\Delta(r,2)=2^{r/2-5}+2.
\end{equation}
\end{theorem}

\begin{proof}
  We apply Construction QM$_3^2$ of Theorem \ref{th5:-3}.  Let the starting code $C_0$  be the $[51,41,3]_22$ code $C_{KR}$ with the parity check matrix $\Hb_{KR}=\Hb_0$, see Theorems \ref{th4:KR code}, \ref{th5:PartitionKR}. So, $n_0=51$, $r_0=10$, $\Ps_0=\Ps_{KR}$, $p(\Hb_0;\Ps_0)=11$. We take $m=6,7,8$ and obtain the needed codes, see \eqref{eq11:CovDensity}, \eqref{eq5:-3res}, \eqref{eq5:notEvenR=2New2}. The decrease  $\Delta(r,2)$ follows from \eqref{eq4:KnwFamR=2even r}, \eqref{eq5:Delta}.
\end{proof}

In Tables \ref{tab5:LengFun r32R=2} and \ref{tab5:LengFun33r64R=2}, the parameters $n,r$, and covering density \eqref{eq11:CovDensity} of the best (as far as the authors know) binary linear covering $[n,n-r]_22$ codes of covering radius $R=2$ and redundancy (codimension) $2\le r\le64$  are written. The values of $n$ give an upper bound on the length function so that $\ell_2(r,2)\le n$; the star * notes the exact value of $\ell_2(r,2)$. The known result of \cite{KaikRoseADSlike2003} that is not presented in the book \cite {CHLL-bookCovCod} is noted by $\bullet$. The new results, obtained in this paper using Theorems \ref{th5:-3}, \ref{th5:PartitionKR}, \ref{th5:r=18,28}--\ref{th5:n0Psi}, and Lemma \ref{lem5:Delta}, are written in bold font and also are noted by the big star $\bigstar$.

\begin{table}[htbp]
\caption{The parameters of the best binary linear $[n,n-r]_22$ covering codes, $r\le32$,
such that $\ell_2(r,2)\le n$; the values of $n$ follow from $\phi(r),f(r)$ \eqref{eq4:notatKnwR=2even}--\eqref{eq4:KnwFamR=2odd r} (the known values) and $\Phi(r),\widehat{\Phi}(r)$ \eqref{eq5:notEvenR=2New1}--\eqref{eq5:r=28}, \eqref{eq5:r=22,24,26}, \eqref{eq5:QM1ares} (the new ones, bold font); $\Delta=\Delta(r,2)$}
 \centering
  \begin{tabular}
  {r|c|r|c|l|c|c||c|c|c}\hline
$r$&for $n$ &\multicolumn{1}{c|}{$n$}&references&density&$p(\Hb)$&$\Delta$&$n_0$&$m$&Th.\\\hline
2 &&$2^{*}$&\cite{CHLL-bookCovCod,GrahSlo}&1      &    &&&\\
3 &$f(3)$&$4^{*}$&\cite{CHLL-bookCovCod,GrahSlo}&1.37500& &&&\\
4 &&$5^{*}$&\cite{CHLL-bookCovCod,GrahSlo}&1      & &&&\\
5 &$f(5)$&$9^{*}$&\cite{CHLL-bookCovCod,GrahSlo}&1.43750&    &&&\\
6 &&$13^{*}$&\cite{CHLL-bookCovCod,GrahSlo,Struik1994R=2-3Bin}&1.43750&    &&&\\
7 &$f(7)$&$19^{*}$ &\cite{DDL-IEEE,GabDavTombR=2,GrahSlo,Struik1994R=2-3Bin}&1.49219&&&&\\\hline
8 &$\phi(8)$&26  &\cite{BrPl1990,DDL-IEEE,GabDavTombR=2}&1.37500&&&&\\
9 &$f(9)$&39     &\cite{GabDavTombR=2}    &1.52539  &&&&\\
10&$\Phi(10)$&    $\bullet51$        &\cite{KaikRoseADSlike2003}   &1.29590&$\textbf{11}\bigstar$&&&&\ref{th5:PartitionKR}\\
11&$f(11)$& 79    &\cite{DDL-IEEE,GabDavTombR=2}&1.54346& &       &&\\
12&$\phi(12)$&107& \cite{DDL-IEEE}&1.41089&&&&\\
13&$f(13)$&159   & \cite{GabDavTombR=2}    &1.55286&&        &&\\
14&$\phi(14)$&215& \cite{DDL-IEEE}&1.41730&&&&\\
15&$f(15)$&319    &\cite{DDL-IEEE,GabDavTombR=2}&1.55765&&          &&\\
16&$\phi(16)$&431 &\cite{DDL-IEEE}&1.42055&&&&\\
17&$f(17)$&639    &\cite{GabDavTombR=2}&1.56007&&            &&\\\hline
18&$\mathbf{\Phi(18)}$&\textbf{831}&$\bigstar$&\textbf{1.31873}&$\mathbf{2^5+1}$&$\mathbf{2^{5}}$&51&4&\ref{th5:r=18,28} \\
19&$f(19)$&1279 &\cite{DDL-IEEE,GabDavTombR=2}&1.56128&&&&&\\

20&$\mathbf{\Phi(20)}$&\textbf{1663}&$\bigstar$&\textbf{1.31952}&$\mathbf{2^6+1}$&$\mathbf{2^{6}}$&51&5&\ref{th5:n0Psi}\\

21&$f(21)$& 2559 &\cite{DDL-IEEE,GabDavTombR=2}&1.56189&  &&&\\
22&$\mathbf{\widehat{\Phi}(22)}$&\textbf{3389}&$\bigstar$&\textbf{1.36956}&&$\textbf{66}$&51&6&\ref{th5:-3} \\
23&$f(23)$& 5119 &\cite{DDL-IEEE,GabDavTombR=2}&1,56219&  &&&\\
24&$\mathbf{\widehat{\Phi}(24)}$&\textbf{6781}&$\bigstar$&\textbf{1.37057}&&$\textbf{130}$&51&7&\ref{th5:-3} \\
25&$f(25)$&10239 &\cite{GabDavTombR=2}&1.56235&& &&     \\
26&$\mathbf{\widehat{\Phi}(26)}$&\textbf{13565} &$\bigstar$&\textbf{1.37107}&&$\textbf{258}$&51&8&\ref{th5:-3}\\
27&$f(27)$&20479&\cite{DDL-IEEE,GabDavTombR=2}&1.56242&&&&\\
28&$\mathbf{\Phi(28)}$&\textbf{26623}&$\bigstar$ &\textbf{1.32026}&$\mathbf{2^6+2}$&$\mathbf{2^{10}}$&$\Phi(18)$&5&\ref{th5:r=18,28}\\
29&$f(29)$&40959&\cite{DDL-IEEE,GabDavTombR=2}&1.56246&&&&\\
30&$\mathbf{\Phi(30)}$&\textbf{53247}&$\bigstar$ &\textbf{1.32029}&$\mathbf{2^7+1}$&$\mathbf{2^{11}}$&$\Phi(18)$&6&\ref{th5:n0Psi}\\
31&$f(31)$&81919&\cite{DDL-IEEE,GabDavTombR=2}&1.56248&&&&\\
32&$\mathbf{\Phi(32)}$&\textbf{106495}&$\bigstar$&\textbf{1.32030}&$\mathbf{2^8+1}$&$\mathbf{2^{12}}$&$\Phi(18)$&7&\ref{th5:n0Psi}\\\hline
 \end{tabular}
\label{tab5:LengFun r32R=2}
\end{table}

\begin{table}[htbp]
\caption{The parameters of the best binary linear $[n,n-r]_22$ covering codes,
$33\le r\le64$, such that $\ell_2(r,2)\le n$; the values of $n$ follow from $f(r)$ \eqref{eq4:notatKnwR=2odd}, \eqref{eq4:KnwFamR=2odd r} (the known values) and
$\Phi(r)$ \eqref{eq5:notEvenR=2New1}, \eqref{eq5:QM1ares} (the new ones, bold font); $\Delta=\Delta(r,2)$}
 \centering
  \begin{tabular}
{r|c|r|c|r|c|c||c|r}\hline
$r$&for $n$&\multicolumn{1}{c|}{$n$}&refer.&density&$p(\Hb)$&$\Delta$&$n_0$&$m$\\\hline
33&$f(33)$&163839   &\cite{GabDavTombR=2}&1.56249&&&\\
34&$\mathbf{\Phi(34)}$&\textbf{212991}&$\bigstar$&\textbf{1.32031}&$\mathbf{2^8+1}$&$\mathbf{2^{13}}$&$\Phi(20)$&7\\
35&$f(35)$&327679   &\cite{DDL-IEEE,GabDavTombR=2}&1.56250&&&\\
36&$\mathbf{\Phi(36)}$&\textbf{425983}&$\bigstar$&\textbf{1.32031}&$\mathbf{2^9+1}$&$\mathbf{2^{14}}$&$\Phi(20)$&8\\
37&$f(37)$&655359   &\cite{DDL-IEEE,GabDavTombR=2}&1.56250&&&\\
38&$\mathbf{\Phi(38)}$&\textbf{851967}&$\bigstar$&\textbf{1.32031}&$\mathbf{2^{10}+1}$&$\mathbf{2^{15}}$&$\Phi(20)$&9\\
39&$f(39)$&1310719  &\cite{DDL-IEEE,GabDavTombR=2}&1.56250&&&\\
40&$\mathbf{\Phi(40)}$&\textbf{1703935}&$\bigstar$&\textbf{1.32031}&$\mathbf{2^{11}+1}$&$\mathbf{2^{16}}$&$\Phi(20)$&10\\
41&$f(41)$&  2621439&\cite{DDL-IEEE,GabDavTombR=2}&1.56250&&&\\
42&$\mathbf{\Phi(42)}$&\textbf{3407871}&$\bigstar$&\textbf{1.32031}&$\mathbf{2^8+1}$&$\mathbf{2^{17}}$&$\Phi(28)$&7\\
43&$f(43)$& 5242879&\cite{DDL-IEEE,GabDavTombR=2}&1.56250&&&\\
44&$\mathbf{\Phi(44)}$&\textbf{6815743}&$\bigstar$&\textbf{1.32031}&$\mathbf{2^9+1}$&$\mathbf{2^{18}}$&$\Phi(28)$&8\\
45&$f(45)$& 10485759&\cite{DDL-IEEE,GabDavTombR=2}&1,56250&&&\\
46&$\mathbf{\Phi(                                                                                     46)}$&\textbf{13631487}&$\bigstar$ &\textbf{1.32031}&$\mathbf{2^{9}+1}$&$\mathbf{2^{19}}$&$\Phi(30)$&8\\
47&$f(47)$& 20971519&\cite{DDL-IEEE,GabDavTombR=2}&1.56250&&&\\
48&$\mathbf{\Phi(48)}$&\textbf{27262975}&$\bigstar$ &\textbf{1.32031}&$\mathbf{2^{10}+1}$&$\mathbf{2^{20}}$&$\Phi(30)$&9\\
49&$ f(49)$&41943039&\cite{DDL-IEEE,GabDavTombR=2}&1.56250&&&\\
50&$\mathbf{\Phi(50)}$&\textbf{54525951}&$\bigstar$ &\textbf{1.32031}&$\mathbf{2^{10}+1}$&$\mathbf{2^{21}}$&$\Phi(32)$&9\\
51&$f(51)$& 83886079&\cite{DDL-IEEE,GabDavTombR=2}&1.56250&&&\\
52&$\mathbf{\Phi(52)}$&\textbf{109051903}&$\bigstar$&\textbf{1.32031}&$\mathbf{2^{11}+1}$&$\mathbf{2^{22}}$&$\Phi(32)$&10\\

53&$f(53)$&167772159&\cite{DDL-IEEE,GabDavTombR=2}&1.56250&&&\\
54&$\mathbf{\Phi(54)}$&\textbf{218103807}&$\bigstar$&\textbf{1.32031}&$\mathbf{2^{11}+1}$&$\mathbf{2^{23}}$&$\Phi(34)$&10\\
55&$f(55)$&335544319&\cite{DDL-IEEE,GabDavTombR=2}&1.56250&&&\\
56&$\mathbf{\Phi(56)}$&\textbf{436207615}&$\bigstar$&\textbf{1.32031}&$\mathbf{2^{12}+1}$&$\mathbf{2^{24}}$&$\Phi(34)$&11\\

57&$f(57)$&671088639&\cite{DDL-IEEE,GabDavTombR=2}&1.56250&&&\\
58&$\mathbf{\Phi(58)}$&\textbf{872415231}&$\bigstar$&\textbf{1.32031}&$\mathbf{2^{12}+1}$&$\mathbf{2^{25}}$&$\Phi(36)$&11\\
59&$f(59)$&1342177279&\cite{DDL-IEEE,GabDavTombR=2}&1.56250&&&\\
60&$\mathbf{\Phi(60)}$&\textbf{1744830463}&$\bigstar$&\textbf{1.32031}&$\mathbf{2^{13}+1}$&$\mathbf{2^{26}}$&$\Phi(36)$&12\\

61&$f(61)$&2684354559&\cite{DDL-IEEE,GabDavTombR=2}&1.56250&&&\\
62&$\mathbf{\Phi(62)}$&\textbf{3489660927}&$\bigstar$&\textbf{1.32031}&$\mathbf{2^{13}+1}$&$\mathbf{2^{27}}$&$\Phi(38)$&12\\

63&$f(63)$&5368709119&\cite{DDL-IEEE,GabDavTombR=2}&1.56250&&&\\
64&$\mathbf{\Phi(64)}$&\textbf{6979321855}&$\bigstar$&\textbf{1.32031}&$\mathbf{2^{14}+1}$&$\mathbf{2^{28}}$&$\Phi(38)$&13\\\hline
 \end{tabular}
  \label{tab5:LengFun33r64R=2}
\end{table}

In the tables, for the new codes we give also the size $p(\Hb)$ of a 2-partition of a parity check matrix $\Hb$  and the decrease $\Delta(r,2)$ of the known upper bounds on $\ell_2(r,2)$, provided by the new results, see Theorems \ref{th5:-3}, \ref{th5:PartitionKR}, \ref{th5:r=18,28}--\ref{th5:R=2NewInfFam} and their  Proofs.
Also, in Table \ref{tab5:LengFun r32R=2}, we write the important new result $p(\Hb_{KR};\Ps_{KR})=11$, obtained in Theorem \ref{th5:PartitionKR}.

The known $\phi(r)$, $f(r)$, and new $\boldsymbol{\Phi(r)}$, $\boldsymbol{\widehat{\Phi}(r)}$ values of $n$ see in \eqref{eq4:notatKnwR=2even}--\eqref{eq4:KnwFamR=2even r}, \eqref{eq4:KnwFamR=2odd r} and \eqref{eq5:notEvenR=2New1}--\eqref{eq5:r=28}, \eqref{eq5:r=22,24,26}, \eqref{eq5:InfFamR=2New}. Also, for the known $n$ we give the references (the columns ``references'' and  ``refer.'').

The new results are obtained using Theorems \ref{th5:-3}, \ref{th5:PartitionKR}, \ref{th5:r=18,28}--\ref{th5:n0Psi} (Table \ref{tab5:LengFun r32R=2}) and only Theorem \ref{th5:n0Psi}  (Table \ref{tab5:LengFun33r64R=2}). It is important to obtain the new results to do a convenient choice of the starting $[n_0,n_0-r_0]_22$ codes for these theorems and provide the needed 2-partitions.  In the tables, for the starting code we give the value $n_0$ that allows us and readers to easily find $r_0$; also, we write the value of $m$ used. Finally, only in Table \ref{tab5:LengFun r32R=2}, in column ``Th.'' we note the theorems used.

The tables illustrate the iterative process to obtain the infinite family of Theorem \ref{th5:R=2NewInfFam}.

\begin{theorem}\label{th5:n0Psi}
 $\mathbf{Construction~QM_6^2}$. Let $\Phi(2t)$ be as in \eqref{eq5:notEvenR=2New1} and let $\Delta(r,2)$ be the decrease of the known upper bounds on $\ell_2(r,2)$. In Construction \emph{QM}$_2^2$ of Theorem \emph{\ref{th4:DDLR=2}}, let  the starting code $C_0$ be an $[n_0,n_0-2t_0,3]_22$ code with $n_0=\Phi(2t_0)$, $r_0=2t_0$, and a parity check matrix $\Hb_0$ with $p(\Hb_0)\le2^{\lambda_0}+2$. For $r=2t_0+2m$, $n=2^mn_0+2^m-1$, we define a new $[n,n-r,d_C]_2R_C$ code $C$ by a parity check
$r\times n$ matrix $\Hb_C$ of the form \eqref{eq32:QM-H}--\eqref{eq32:varD} with the following indicator set $\Bs$, parameter $m$, and submatrix~$\D$:
\begin{equation}
\emph{QM}_6^2:~\Bs=\F_{2^m},~t_0\ge m\ge \lambda_0+1,~
 \D=\D_1(2). \label{eq5:QM1ainp}
\end{equation}
Then the new code $C$ is an $[n, n - r,3]_22$ code with the parameters:
\begin{equation}
\emph{QM}_6^2:~R=2,~r = 2t,~t=t_0+m,~n=\Phi(2t),~p(\Hb_C)\le 2^{m+1}+1,~\Delta(r,2)=2^{t-4}.\label{eq5:QM1ares}
\end{equation}
\end{theorem}

\begin{proof}
For $n_0=\Phi(2t_0)$ and $p(\Hb_0)\le2^{\lambda_0}+2$, the condition $n_0\ge2^m\ge p(\Hb_0)$  of \eqref{eq4:CaseA2inp} holds if $t_0\ge m\ge \lambda_0+1$. So, we can use the results of
\eqref{eq4:CaseA2res}. In particular, we have $n=2^m(\Phi(2t_0)+1)-1=26\cdot2^{t_0-4+m}-1=\Phi(2t_0+2m)=\Phi(2t)$. The value of $\Delta(r,2)$ follows from \eqref{eq4:KnwFamR=2even r}, \eqref{eq5:Delta}. The minimum distance $d_C=3$ follows from $\D$.
\end{proof}

\begin{theorem}\label{th5:R=2NewInfFam}
Let the asymptotic covering density $\overline{\mu}(R)$ be as in \eqref{eq11:liminfdens}. Let $\Phi(r)$ be as in \eqref{eq5:notEvenR=2New1}.
There exists a new infinite family of binary linear  $[n,n-r,3]_22$ codes
with a parity check matrix $\Hb$ of the form \eqref{eq32:QM-H}--\eqref{eq32:varD}, and parameters as in \eqref{eq5:InfFamR=2New}.
\begin{align}\label{eq5:InfFamR=2New}
&R=2,~r=2t,~r=10,18,20,\T{ and }r\ge28,~t=5,9,10,\T{ and }t\ge14, \db\\
&n=26\cdot2^{r/2-4}-1= 26\cdot2^{t-4}-1=\Phi(r),~\overline{\mu}(2)\thickapprox26^2/2^9\thickapprox1.32031;\notag\db\\
&\ell_2(r,2)=\ell_2(2t,2)\le26\cdot2^{r/2-4}-1= 26\cdot2^{t-4}-1=\Phi(r),~\Delta(r,2)=2^{r/2-4}. \notag
\end{align}
Let $p(\Hb)$ be the size of a $2$-partition of the parity check matrix $\Hb$. For even $r=10,18,20$, $28\le r\le40$, the sizes $p(\Hb)$ are given in Tables \emph{\ref{tab5:LengFun r32R=2}, \ref{tab5:LengFun33r64R=2}}. For even $r\ge42$, we have
\begin{equation}\label{eq5:InfFamR=2NewLenFunLambda}
 \lambda=\lfloor r/4\rfloor-2,~p(\Hb)=2^\lambda+1,~\lambda+1=t-(\lceil r/4\rceil+1)\le t-12.
\end{equation}
\end{theorem}

\begin{proof}
We prove the assertions via an iterative process; it is illustrated by Tables \ref{tab5:LengFun r32R=2}, \ref{tab5:LengFun33r64R=2}.

\textbf{(i)} For $r=10$, $t=5$, we have \cite{KaikRoseADSlike2003} the $[51,41,3]_22=[\Phi(10),\Phi(10)-10]$ code $C_{KR}$ with the parity check matrix $\Hb_{KR}$, see  Theorem \ref{th4:KR code}. The new 2-partitions $\Ps_{KR}$ \eqref{eq5:PKR} and $\Ps_{KR}^*$ \eqref{eq5:PKR*} of $\Hb_{KR}$, obtained in Theorem \ref{th5:PartitionKR} and Proof of Theorem \ref{th5:r=18,28}, are very important for the next iterative process.

\textbf{(ii)} For $r=18,28$, Theorem \ref{th5:r=18,28} with \eqref{eq5:r=18}, \eqref{eq5:r=28} obtains $[\Phi(18),\Phi(18)-18,3]_22$ and $[\Phi(28),\Phi(28)-28,3]_22$ codes with $p(\Hb)=2^5+1$ and $p(\Hb)=2^6+2$, respectively.

\textbf{(iii)} For even $r=20$ and $r\ge30$, we iteratively use Construction QM$_6^2$ of Theorem~\ref{th5:n0Psi}
 with \eqref{eq5:QM1ainp}, \eqref{eq5:QM1ares},
choosing convenient starting codes $C_0$ and parameters $m$.

\textbf{(iii-1)} As the starting code $C_0$ we take the $[51,41,3]_22$ code $C_{KR}$ with $n_0=\Phi(10)$, $r_0=10$, $t_0=5$.  By Theorem \ref{th5:PartitionKR}, we have $p(\Hb_0)=11<2^4+1$,  $\lambda_0=4$. The condition $t_0\ge m\ge \lambda_0+1$ from \eqref{eq5:QM1ainp} holds if $m=4,5$. We take $m=5$. By \eqref{eq5:QM1ares}, we obtain a $[\Phi(20),\Phi(20)-20]_22$ code with $r=20$, $t=10$, $p(\Hb)=2^6+1$, $\Delta(20,2)=2^6$.

\textbf{(iii-2)} Let $C_0$ be the $[\Phi(18),\Phi(18)-18,3]_22$ code with $r_0=18$, $t_0=9$, $p(\Hb)=2^5+1$. We have $t_0\ge m\ge \lambda_0+1$ if $m=6,7,8,9$. We take $m=6,7$. By \eqref{eq5:QM1ares}, we obtain  $[\Phi(r),\Phi(r)-r,3]_22$ codes with $r=30,32$, $t=r/2$, $p(\Hb)=2^{m+1}+1$, $\Delta(r)=2^{t-4}$.

\textbf{(iii-3)} Let $C_0$ be the $[\Phi(20),\Phi(20)-20,3]_22$ code with $r_0=20$, $t_0=10$, $p(\Hb)=2^6+1$. We have $t_0\ge m\ge \lambda_0+1$  if $m=7,8,9,10$. For these $m$, by \eqref{eq5:QM1ares}, we obtain $[\Phi(r),\Phi(r)-r,3]_22$  codes with $r=34,36,38,40$, $t=r/2$, $p(\Hb)=2^{m+1}+1$, $\Delta(r,2)=2^{t-4}$.

\textbf{(iii-4)}  Let $C_0$ be the $[\Phi(28),\Phi(28)-28,3]_22$ code with $r_0=28$, $t_0=14$, $p(\Hb)=2^6+2$. We have $t_0\ge m\ge \lambda_0+1$ if $m=7,8,\ldots,14$. We take $m=7,8=t_0-7,t_0-6$. By \eqref{eq5:QM1ares}, we obtain  $[\Phi(r),\Phi(r)-r,3]_22$ codes with $r=42,44$, $t=r/2$, $\lambda=\lfloor r/4\rfloor-2=t-13$, $p(\Hb)=2^\lambda+1$, $\lambda+1=t-(\lceil r/4\rceil+1)=t-12$, $\Delta(r,2)=2^{t-4}$.

\textbf{(iii-5)} We consider 8 the starting $[\Phi(r_0),\Phi(r_0)-r_0,3]_22$ codes $C_0$ with $r_0=2t_0=30,32,\ldots,44$, $t_0=15,16,\ldots,22$. By \eqref{eq5:QM1ainp}, the allowed values of $m$ satisfy $t_0\ge m\ge \lambda_0+1$, where $\lambda_0$ depends on the specific code. By above, for the considered values of $t_0$ we have, respectively, the following values of $\lambda_0+1$: $t_0-7,t_0-7,t_0-8,t_0-8,t_0-8,t_0-8,t_0-12,t_0-12$. For all $C_0$, we use $m=t_0-7,t_0-6$, and by \eqref{eq5:QM1ares}, obtain 16  new $[\Phi(r),\Phi(r)-r,3]_22$ codes with even $r=46,48,\ldots,74,76$,
$t=r/2$, $\lambda=\lfloor r/4\rfloor-2$, $p(\Hb)=2^\lambda+1$, $\Delta(r,2)=2^{r/2-4}$, $\lambda+1=t-(\lceil r/4\rceil+1)\le t-12$, see Table \ref{tab5:LengFun33r64R=2} for $r\le64$.

Obviously, we can continue this iterative process infinitely obtaining the infinite code family of \eqref{eq5:InfFamR=2New}.
\end{proof}

\section{The known results on binary linear codes of covering radius 3}\label{sec6:R=3known}

\begin{theorem}\label{th6:DDLD95R=3} \emph{\cite{Dav90PIT}, \cite[Theorem 4.1]{DDL-IEEE}}
$\mathbf{Constructions~QM_1^3 - QM_4^3}$. In Construction \emph{QM} of Section~$\ref{subsec32:QM}$, let the starting code $C_0$  be an $[n_0,n_0-r_0]_23,\ell_0$ code
with a parity check matrix $\Hb_0$ of the form \eqref{eq32:H0}. We define a new $[n,n-r]_2R_C$ code $C$ with $r=r_0+3m$, $n=2^mn_0+N_m$ by a parity check $r\times n$ matrix $\Hb_C$ of the form \eqref{eq32:QM-H}--\eqref{eq32:Hcb2m}. Let the indicator set $\Bs$, parameter $m$, and $r\times N_m$ submatrix $\D$ be chosen by one of the ways \eqref{eq6:ell0=0inp}--\eqref{eq6:any ell0CSIinp} which set specific  Constructions
\emph{QM}$_1^3$--\emph{QM}$_4^3$.
\begin{align}
&\emph{QM}_1^3: ~\ell_0=0,~\Bs\subseteq\F_{2^m}^*,\,2^m-1\ge p(\Hb_0,0),\,
 \D=\D_4.\label{eq6:ell0=0inp}\db\\
&\emph{QM}_2^3: ~\ell_0=1,\,\Bs\subseteq\F_{2^m},\,2^m\ge p(\Hb_0,1),\,
 \D=\D_2(3).\label{eq6:ell0=1inp}\db\\
 &\emph{QM}_3^3: ~\ell_0=2,~\Bs\subseteq\F_{2^m}\cup\{*\},~2^m+1\ge p(\Hb_0,2),~
 \D=\D_3.\label{eq6:ell0=2inp}\db\\
 &\emph{QM}_4^3:~\forall\ell_0,~\Bs=\F_{2^m}\cup\{*\},~n_0\ge2^m+1\ge p(\Hb_0,\ell_0),~
 \D=\D_3.\label{eq6:any ell0CSIinp}
\end{align}
Then the new code $C$ is an $[n, n - r]_23,\ell$ code with $r = r_0 + 3m$, and parameters:
\begin{align}
 &\emph{QM}_1^3: ~n=2^m(n_0+1)+n_{2m}-1,~\ell\ge1.\label{eq6:ell0=0res}\db\\
 &\emph{QM}_2^3: ~n=2^m(n_0+2)-2,~\ell=2,~p(\Hb_C,1)\le p(\Hb_0,1)+2.\label{eq6:ell0=1res}\db\\
 &\emph{QM}_3^3: ~n=2^m(n_0+1)-1,~\ell=2,~p(\Hb_C,2)\le p(\Hb_0,2)+1.\label{eq6:ell0=2res}\db\\
 &\emph{QM}_4^3: ~n=2^m(n_0+1)-1,~\ell=2\T{ for }m\ge2,~p(\Hb_C,\ell_0)\le2^{m+1}+3.\label{eq6:any ell0CSIres}
\end{align}
 \end{theorem}

\begin{notation}\label{not6:psi}
 We introduce the following notations:
\begin{align*}
&\varphi(r)\triangleq3\cdot2^{(r-1)/3}-1=3\cdot2^{t-1}-1\T{ for }r=3t-2;\db\\
&\gamma(r)\triangleq821\cdot2^{(r-26)/3}-1=821\cdot2^{t-9}-1\T{ for }r=3t-1;\db\\
&\psi(r)\triangleq144\cdot2^{(r-18)/3}-1=144\cdot2^{t-6}-1\T{ for }r=3t.
\end{align*}
\end{notation}

\begin{theorem}\label{th6:KnownFamR=3}
\emph{\cite{DDL-IEEE,DDL-ACCT2-1990,DDL-ACCT3-1992,DavCovNewConstrCovCod}} Let the asymptotic density $\overline{\mu}(R)$ be as in \eqref{eq11:liminfdens}. There are infinite families of $[n,n-r]_23,\ell$ codes with growing $r$ and the following  parameters:
\begin{align}
&r=3t-2\ge22,~t\ge8,\T{and } r=4,7,16,~n=3\cdot2^{(r-1)/3}-1=3\cdot2^{t-1}-1=\varphi(r),\notag\db\\
&\ell=2\T{ if }r\ge16,~\overline{\mu}(3)\thickapprox2.25; ~\ell_2(3t-2,3)\le3\cdot2^{t-1}-1   \emph{\cite[Example 4.3]{DDL-IEEE}}.\notag\db\\
&r=3t-1\ge14,n=1024\cdot2^{(r-26)/3}-1,\ell=2,\overline{\mu}(3)\thickapprox2.6667\notag\\
&\emph{\cite[p.\ 56, Table]{DDL-ACCT2-1990},\cite[Example 4.1]{DDL-IEEE}}.\notag\db\\
&r=3t-1\ge26,n=822\cdot2^{(r-26)/3}-2,\ell\ge1,\overline{\mu}(3)\thickapprox1.3794\notag\\
&\emph{\cite[p.\ 56, Table]{DDL-ACCT2-1990},\cite[Example 4.2]{DDL-IEEE}}.\notag\db\\
&r=3t-1\ge41,~t\ge14,~n=821\cdot2^{(r-26)/3}-1=821\cdot2^{t-9}-1=\gamma(r),
\label{eq6:knownfam3t-1}\db\\
&\overline{\mu}(3)\thickapprox1.3744;
~\ell_2(3t-1,3)\le821\cdot2^{t-9}-1\T{ if }t\ge14\emph{\cite[p.\ 53]{DDL-ACCT3-1992}, \cite[Example 4.2]{DDL-IEEE}}.\notag\db\\
&r=3t\ge18,n=155\cdot2^{(r-18)/3}-2,\ell\ge1,\overline{\mu}(3)\thickapprox2.3576\notag\\
&\emph{\cite[p.\ 56, Table]{DDL-ACCT2-1990},\cite[Example 4.4]{DDL-IEEE}};\notag\db\\
&r=3t\ge27,~n=152\cdot2^{(r-18)/3}-1,~\ell=2,~\overline{\mu}(3)\thickapprox2.2327 \emph{\cite[Example 4.5]{DDL-IEEE}}.\notag\db\\
&r=3t\ge30,t\ge10,\T{ and }r=15,\,n=144\cdot2^{(r-18)/3}-1=144\cdot2^{t-6}-1=\psi(r),
\notag\db\\
&\ell=2,~\overline{\mu}(3)\thickapprox1.8984,~\ell_2(3t,3)\le144\cdot2^{t-6}-1\T{ if }t=5, t\ge10\emph{\cite[Examples 5,\,8]{DavCovNewConstrCovCod}}.\notag
 \end{align}
\end{theorem}

\begin{theorem}\label{th6:OK189}\emph{\"{O}sterg{\aa}rd, Kaikkonen \cite[Table 2]{OstKaikUpBndBin1998}}
Let $C_{OK}$ be the binary $[18,9]_2$ code with the parity check matrix
$\Hb_{OK}=[\Ib_{9}\Mb_{OK}]$, where $\Mb_{OK}$ is the $9\times9$ binary matrix with the following columns in
hexadecimal notation:
\begin{align}\label{eq6:HOK}
&\Mb_{OK}=[\mathrm{1A0,174,A5,173,17,E8,9,18D,ICE}].
\end{align}
Then $C_{OK}$ is a $[18,9]_23$ code of covering radius $3$.
\end{theorem}

\section{New results on binary linear codes of covering radius 3}\label{sec7:R=3new}

\begin{theorem}\label{th7:PartitionOK}
  Let $\Hb_{OK}=[\Ib_9\Mb_{OK}]$ be the binary $9\times18$ matrix with columns of the $9\times9$ submatrix $\Mb_{OK}$ given in \eqref{eq6:HOK}. Let the columns of $\Hb_{OK}$ be numbered from left to right. Let the partition $\Ps_{OK}$ of the column set of $\Hb_{OK}$ into $11$ subsets  be as follows:
  \begin{equation*}
 \Ps_{OK}\triangleq\{1,2,4\},\{3\},\{5,8\},\{6,17\},\{7,10\},\{11,14\},\{12\},\{13,18\},\{15\},\{9\},\{16\},\notag
  \end{equation*}
where column numbers for every subset are written. Then, $\Ps_{OK}$ is a $(3,1)$-partition of $\Hb_{OK}$, i.e. $p(\Hb_{OK},1;$ $\Ps_{OK})=11$ and
the $[18,9]_23$ code $C_{OK}$ with the parity check matrix $\Hb_{OK}$ of Theorem \emph{\ref{th6:OK189}} is an $[18,9,3]_23,1$ code with $p(\Hb_{OK},1)\le11$.
\end{theorem}

\begin{proof}
  We obtained $\Ps_{OK}$ by computer search. The sum of the columns 6, 9, 16 is: $(0,0,0,0,0,1,0,0,0)^{tr}+(0,0,0,0,0,0,0,0,1)^{tr}+(0,0,0,0,0,1,0,0,1)^{tr}=\zb_9$. So, 3 linear dependent columns are placed in distinct subsets.
\end{proof}

\begin{theorem}\label{th7:303R=3}
There exists a new $[303,282,3]_23,2$ code with a parity check matrix, admitting a $(3,1)$-partition into $35$ subsets, and with $\Delta(21,3)=5$.
\end{theorem}

\begin{proof}
 We use Construction QM$_4^3$ of Theorem \ref{th6:DDLD95R=3}. As the code $C_0$ we take the $[18,9,3]_23,1$ code $C_{OK}$ with the (3.1)-partition $\Ps_{OK}$, see Theorems \ref{th6:OK189}, \ref{th7:PartitionOK}. We have $n_0=18>2^4+1>p(\Hb_0,1)=11$ that allows us to take $m=4$ in \eqref{eq6:any ell0CSIinp}. By \eqref{eq6:any ell0CSIres}, we obtain a new $[n,n-21]_23,2$ code $C$ with $n=303$ and $p(\Hb_C,1)\le2^{m+1}+3$. By Theorem \ref{th6:KnownFamR=3}, the known value is $n=155\cdot2^{(21-18)/3}-2=308$.
\end{proof}

\begin{theorem}\label{th7:ell0=0R=3}
$\mathbf{Construction~QM_5^3.}$ In Construction \emph{QM} of Section~$\ref{subsec32:QM}$, let the starting code $C_0$  be an $[n_0,n_0-r_0]_23,0$ code
with a parity check matrix $\Hb_0$ of the form \eqref{eq32:H0} admitting a $(3,0)$-partition $\Ps_0$
into $p(\Hb_0,0)$ subsets. Let $\Vc_{2m}$ be an $[n_{2m},n_{2m}-2m]_22$  code with a parity check matrix $\Hcb_{2m}$ admitting a $2$-partition $\Ps_{2m}$ into $p(\Hcb_{2m})$ subsets.
We define a new $[n,n-r]_2R_C,\ell_C$ code $C$ by a parity check
$r\times n$ matrix $\Hb_C$ of the form \eqref{eq32:QM-H}--\eqref{eq32:Hcb2m}, where the indicator set $\Bs$, parameter $m$, and submatrix $\D$ are as follows:
\begin{equation}
\emph{QM}_5^3:~\Bs\subseteq\F_{2^m}^*,~2^m-1\ge p(\Hb_0,0),~m\ge2,~
 \D=\D_4.\label{eq7:ell0=0inpR=3}
\end{equation}
Then the new code $C$ is an $[n, n - r]_23,\ell_C$ code of covering radius $3$ with the parameters:
\begin{align}
& \emph{QM}_5^3:~n=2^m(n_0+1)+n_{2m}-1,~r = r_0 + 3m,~\ell_C\ge1;\label{eq7:ell0=0resR=3}\db\\
&\T{if any column of $\F_2^{2m}$ is a sum of $2$ or $3$ columns of $\Hcb_{2m}$ then }\ell_C=2;\label{eq7:part-ellC=2}\db\\
&p(\Hb_C,0)\le p(\Hb_0,0)+p_{2m}(\Hcb_{2m})+1,~ p(\Hb_C,1)= p(\Hb_C,0)+2,~p(\Hb_C,2)\le n.\label{eq7:part-ellR=3}
\end{align}
If there exists a triple of linearly dependent columns of $\Hb_0$ or $\Hcb_{2m}$  belonging to distinct subsets of the partitions $\Ps_0$ or $\Ps_{2m}$, respectively, then $p(\Hb_C,1)= p(\Hb_C,0)$.
 \end{theorem}

\begin{proof}
We prove the assertions considering the representation as a sum of $\le3$ columns of $\Hb_C$ for an arbitrary column $\Ub=(\pib,\ub_1,\ub_2,\ub_3)^{tr}\in\F_2^{\,r}$, where $\pib\in\F_2^{\,r_0}$, $\ub_1,\ub_2,\ub_3\in\F_2^{\,m}$.

\begin{description}
  \item[(1)] $\pib=\hb_i+\hb_j+\hb_k$; $\hb_i$, $\hb_j$, $\hb_k$ belong to distinct subsets of $\Ps_0$; $\beta_i,\beta_j,\beta_k$ are distinct.

  We find $\xb,\yb,\zbl$ solving the linear system $\beta_i^{\delta-1} \xb+\beta_j^{\delta-1} \yb+\beta_k^{\delta-1} \zbl=\ub_\delta$, $\delta=1,2,3$.
  Now $\Ub= (\hb_i,\xb,\beta_i \xb,\beta_i^2 \xb)^{tr} +(\hb_j,\yb,\beta_j \yb,\beta_j^2 \yb)^{tr}+(\hb_k,\zbl,\beta_k \zbl,\beta_k^2 \zbl)^{tr}$.

  \item[(2)] $\pib=\hb_i+\hb_j$; $\hb_i$ and $\hb_j$ belong to distinct subsets of $\Ps_0$; $\beta_i\ne\beta_j$, $\beta_i,\beta_j\ne0$.

  We find $\xb,\yb$ solving the linear system $\beta_i^{\delta-1} \xb+\beta_j^{\delta-1} \yb=\ub_\delta$,  $\delta=2,3$.
  Now $\Ub=(\hb_i,\xb,\beta_i \xb,\beta_i^2 \xb)^{tr} +(\hb_j,\yb,\beta_j \yb,\beta_j^2 \yb)^{tr}+(\zb_{r_0,}\wb,\zb_{2m})^{tr}$ where $\wb=\xb+\yb+\ub_1\in\Wb_m$. If $\xb+\yb=\ub_1$ then the addend $(\zb_{r_0,}\wb,\zb_{2m})^{tr}$ is absent.

\item[(3)] $\pib=\hb_j$.

  Let $(\beta_j\ub_1,\beta_j^2\ub_1)^{tr}\ne(\ub_2,\ub_3)^{tr}$. Then $\Ub=(\hb_j,\ub_1,\beta_j\ub_1,\beta_j^2\ub_1)^{tr} +\boldsymbol{S}$, where $\boldsymbol{S}$ is a single column or a sum of two columns of $\Hcb_{2m}$ belonging to distinct subsets of $\Ps_{2m}$.

Let $(\beta_j\ub_1,\beta_j^2\ub_1)^{tr}=(\ub_2,\ub_3)^{tr}$. For the new code $C$, we consider $\ell_C\in\{0,1,2\}$. For $\ell_C\in\{0,1\}$, it is sufficient to put $\Ub=(\hb_j,\ub_1,\beta_j\ub_1,\beta_j^2\ub_1)^{tr}$.
  To provide $\ell_C=2$, we find distinct elements $\xi_{i_1},\xi_{i_2},\xi_{i_3}$ of $\F_{2^m}$, see \eqref{eq32:Aj}, \eqref{eq32:xi}, such that
  $\xi_{i_1}+\xi_{i_2}+\xi_{i_3}=\ub_1$ (the needed $\xi_{i_u}$ exist for $m\ge2$). Now  $\Ub=\sum_{u=1}^3(\hb_j,\xi_{i_u},\beta_j\xi_{i_u},\beta_j^2\xi_{i_u})^{tr}$.

\item[(4)] $\pib=\zb_{r_0}$.

  The last $3m$ rows of $\D$ form a parity check matrix of the direct sum \cite{GrahSlo} of the $[2^m-1,2^m-1-m,3]_21$ Hamming code and the code $\Vc_{2m}$ that gives a code of covering radius 3. We
consider the representation of $\Ub$ as a sum of columns of $\D$; this helps us to estimate the cardinality of  the $(3,\ell)$-partitions of $\Hb_C$.

Let $\ub_1=\wb\in\Wb_m$.

If $(\ub_2, \ub_3)^{tr}\ne(\zb_{2m})^{tr}$, then $\Ub=(\zb_{r_0},\wb,\zb_{2m})^{tr} +\boldsymbol{S}$, where $\boldsymbol{S}$ is a single column or a sum of two columns of $\Hcb_{2m}$ belonging to distinct subsets of $\Ps_{2m}$.

If $(\ub_2, \ub_3)^{tr}=(\zb_{2m})^{tr}$, then for $\ell_C\in\{0,1\}$, it is sufficient to put $\Ub=(\zb_{r_0},\wb,\zb_{2m})^{tr} $.
  To provide $\ell_C=2$, we find $\wb_1,\wb_2\in\Wb_m$ such that $\wb_1+\wb_2=\wb$ (obviously, the needed $\wb_i$ exist). Now  $\Ub=\sum_{i=1}^2(\zb_{r_0},\wb_i,\zb_{2m})^{tr}$.

Let $\ub_1=\zb_m$.

If $(\ub_2, \ub_3)^{tr}$ is not a column of $\Hcb_{2m}$, then $(\ub_2, \ub_3)^{tr}$ is a sum of 2 columns of $\Hcb_{2m}$.

If $(\ub_2, \ub_3)^{tr}$ is a column, say $\boldsymbol{\lambda}_{2m}$, of $\Hcb_{2m}$, then for $\ell_C\in\{0,1\}$, it is sufficient to put $\Ub=(\zb_{r_0+m},\boldsymbol{\lambda}_{2m})^{tr} $. To provide $\ell_C=2$, we
should represent $\boldsymbol{\lambda}_{2m}$ as a sum of two or three columns of $\Hcb_{2m}$.
\end{description}

Thus, we have proved that any column $\Ub \in\F_2^{\,r}$  can be represented as  a sum of at most three columns of $\Hb_C$. This means that the covering radius of the new code $C$ is $R_C=3$.

To estimate $p(\Hb_C,0)$ for the matrix $\Hb_C$ of the form \eqref{eq32:QM-H}--\eqref{eq32:Hcb2m}, \eqref{eq7:ell0=0inpR=3},
we partition the columns of $\Hb_C$ (except $\D$) in accordance with the partition
$\Ps_0$ of $\Hb_0$: if columns $\hb_i, \hb_j$ of $\Hb_0$ belong to distinct subsets of
$\Ps_0$, then the columns of submatrices $\Ab(\hb_i,\beta_i),\Ab(\hb_j,\beta_j)$ belong to distinct subsets
of the partition of $\Hb_C$. Also, taking into account the proof above, we partition the columns of $\D$ \eqref{eq7:ell0=0inpR=3} into $p(\Hcb_{2m})+1$ subsets so that the submatrix $\Wb$ forms one subset. This implies $p(\Hb_C,0)\le p(\Hb_0,0)+p(\Hcb_{2m})+1$ \eqref{eq7:part-ellR=3}.

 Now we note that $\Ub=(\zb_{r_0+3m})^{tr}$ can always be represented as a sum of three columns of $\Wb_m$; this is necessary for $\ell_C\ge1$, see Section \ref{subsec31:capsul}. To provide this, we should partition $\Wb_m$ into three subsets, that explains the inequality  $p(\Hb_C,1)\le p(\Hb_C,0)+2$ in~\eqref{eq7:part-ellR=3}. We could avoid the above partition, if there exists a triple of linearly dependent columns of $\Hb_0$ or $\Hcb_{2m}$ belonging to distinct subsets of the corresponding partitions $\Ps_0$ or $\Ps_{2m}$; this explains the case $p(\Hb_C,1)= p(\Hb_C,0)$.

 Finally, by the proof above, if any column of $\F_2^{2m}$ is equal to a sum of two or three columns of $\Hcb_{2m}$ then $\ell_C=2$. For simplicity, in this case we can use the trivial partition and put $p(\Hb_C,2)\le n$.
 \end{proof}

\begin{remark}\label{rem7:ell=2}
  Really, we can prove that if in Construction QM$_5^3$ we have $\ell_C=2$, see \eqref{eq7:part-ellC=2}, then for $m\ge3$, the parity check matrix $\Hb_C$ admits a $(3,2)$-partition into $p(\Hb_C,2)=5p(\Hb_0,0)+n_{2m}+3$ subsets. We do not give the proof to save the space.
 \end{remark}

\begin{lemma}\label{lem7:sum3col}
We consider the $[51,41,3]_22$ code $C_{KR}$ with the parity check matrix $\Hb_{KR}$ of Theorems \emph{\ref{th4:KR code}, \ref{th5:PartitionKR}}.
 Any column of $\F_2^{10}$ is equal to a sum of three columns of $\Hb_{KR}$.
\end{lemma}

\begin{proof}
We have proved the assertion by computer search.
\end{proof}

 Let $\Phi(r)$ be as in \eqref{eq5:notEvenR=2New1}; we introduce the notations.
 \begin{align}
&\Upsilon(r)=\Upsilon(3t-1)\triangleq 819\cdot2^{(r-26)/3}-1=819\cdot2^{t-9}-1,~ r=26,~r=3t-1\ge44.  \label{eq7:notatUpsR3}\db\\
&\widehat{\Upsilon}(r)\triangleq24\cdot2^{(r-11)/3}-1+\Phi(2(r-11)/3)=820\cdot2^{(r-26)/3}-2,~r=38,41.
  \label{eq7:notatUpshatR3}
 \end{align}

\begin{theorem}\label{th7:r=26,38,41}
Let the covering density $\mu$ be as in \eqref{eq11:CovDensity}.
There are $3$ new $[n,n-r,3]_23,\ell$ codes $C$ with a parity check matrix $\Hb_C$ of the form \eqref{eq32:QM-H}--\eqref{eq32:Hcb2m}, \eqref{eq7:ell0=0inpR=3} and parameters as in \eqref{eq7:r=26}--\eqref{eq7:r=41}.
\begin{align}
&r=26,~n=\Upsilon(26)=818,~\mu\thickapprox1.35935,~\ell=2,~p(\Hb_C,0)=p(\Hb_C,1)=35,\label{eq7:r=26}\db\\
&\phantom{r=26,~}p(\Hb_C,2)=\Upsilon(26),~\ell_2(26,3)\le\Upsilon(26),~\Delta(26,3)=2.\notag\db\\
&r=38,~n=\widehat{\Upsilon}(38)=13118,~\mu\thickapprox1.36871,~\ell=1,~p(\Hb_C,0)=57,\label{eq7:r=38}\db\\
&\phantom{r=38,~}p(\Hb_C,1)=59;~\ell_2(38,3)\le\widehat{\Upsilon}(38),~\Delta(38,3)=2^5.\notag\db\\
&r=41,~n=\widehat{\Upsilon}(41)=26238,~\mu\thickapprox1.36902,~\ell=1,~ p(\Hb_C,0)=89,\label{eq7:r=41}\db\\
&\phantom{r=41,~}p(\Hb_C,1)=91;~ \ell_2(41,3)\le\widehat{\Upsilon}(41),~\Delta(41,3)=2^5+1.\notag
\end{align}
\end{theorem}

\begin{proof}
  We apply Construction QM$_5^3$ of Theorem \ref{th7:ell0=0R=3}.  Let the starting code $C_0$  be the perfect $[23,12,7]_23,0$ Golay code with the parity check matrix $\Hb_0$. We use the trivial $(3,0)$-partition of $\Hb_0$. So, $n_0=23$, $r_0=11$, $\ell_0=0$, $p(\Hb_0,0)=23$.

  By \eqref{eq7:ell0=0inpR=3}, we can put $m=5,9,10$ that implies $r=26,38,41$. As $\Vc_{2m}$, we take the $[\Phi(2m),\Phi(2m)-2m]_22$ codes of Theorems \ref{th4:KR code}, \ref{th5:PartitionKR}, \ref{th5:r=18,28}, \ref{th5:n0Psi}, Lemma \ref{lem7:sum3col}, Table \ref{tab5:LengFun r32R=2}, where $\Phi(2m)$ is as in \eqref{eq5:notEvenR=2New1}. We have $\Phi(2\cdot5)=51$, $\Phi(2\cdot9)=831$, $\Phi(2\cdot10)=1663$, $p(\Hcb_{2\cdot5})=11$, $p(\Hcb_{2\cdot9})=2^5+1$, $p(\Hcb_{2\cdot10})=2^6+1$.

  By \eqref{eq7:ell0=0resR=3}, \eqref{eq7:notatUpsR3}, \eqref{eq7:notatUpshatR3}, we obtain $n=\Upsilon(26)$, if $r=26$, and $n=\widehat{\Upsilon}(r)$ if $r=38,41$. By \eqref{eq7:ell0=0resR=3}, $\ell=1$  if $r=38,41$. By \eqref{eq7:ell0=0resR=3}, \eqref{eq7:part-ellC=2}, and Lemma \ref{lem7:sum3col}, $\ell=2$  if $r=26$. By \eqref{eq7:part-ellR=3}, we obtain $p(\Hb_C,\ell)$ as in \eqref{eq7:r=26}--\eqref{eq7:r=41}.

  The decrease $\Delta(r,3)$ follows from $\ell_2(r,3)\le n$ and Theorem \ref{th6:KnownFamR=3}.
\end{proof}

\begin{theorem}\label{th7:newr=3t-1a}
Let the asymptotic covering density $\overline{\mu}(R)$ be as in \eqref{eq11:liminfdens}. Let $\Upsilon(r)$ be as in \eqref{eq7:notatUpsR3}. There exists a new infinite family of binary linear  $[n,n-r,3]_23,2$ codes with a parity check matrix $\Hb$ of the form \eqref{eq32:QM-H}--\eqref{eq32:Hcb2m} and parameters as in \eqref{eq7:InfFamR=3New}, \eqref{eq7:pHbInfFam}.
\begin{align}\label{eq7:InfFamR=3New}
&R=3,~r=3t-1,~r=26\T{ and }r\ge44,~t=9\T{ and }t\ge15,~\ell=2, \db\\
&n=819\cdot2^{(r-26)/3}-1=819\cdot2^{t-9}-1=\Upsilon(r),~\overline{\mu}(3)
\thickapprox1.36433;\notag\db\\
&p(\Hb,0)=p(\Hb,1)=35,~p(\Hb,2)=818\T{ if }r=26;\label{eq7:pHbInfFam}\db\\
&p(\Hb,1)=2^{(r-23)/3}\T{ if }r=44,47,50,53;~p(\Hb,2)=819\T{ if }r\ge56; \notag\db\\
&\ell_2(r,3)=\ell_2(3t-1,3)\le819\cdot2^{t-9}-1=\Upsilon(r),~\Delta(r,3)=2^{(r-23)/3}.
\notag
\end{align}
\end{theorem}

\begin{proof} For $r=26$, see Theorem \ref{th7:r=26,38,41}.
Then we use Construction QM$_4^3$ \eqref{eq6:any ell0CSIinp}, \eqref{eq6:any ell0CSIres}, of Theorem \ref{th6:DDLD95R=3}. As $C_0$ we take the $[\Upsilon(26)=818, \Upsilon(26)-26,3]_23,\ell$ code of Theorem \ref{th7:r=26,38,41}. We consider $\ell=1$ with
$p(\Hb_1,1)=35$ that allows us, by \eqref{eq6:any ell0CSIinp},  to put $m=6,7,8,9$ and obtain the needed codes for $r=44,47,50,53$.

Finally, we use Construction QM$_4^4$ \eqref{eq6:ell0=2inp}, \eqref{eq6:ell0=2res},  of Theorem \ref{th6:DDLD95R=3}, Again as $C_0$ we take the $[\Upsilon(26), \Upsilon(26)-26,3]_23,\ell$ code but now we consider $\ell=2$ with
$p(\Hb_1,1)=818$ that allows us, by \eqref{eq6:ell0=2inp}, to put $m\ge10$ and obtain the needed codes for $r\ge56$.

By \eqref{eq6:knownfam3t-1}, \eqref{eq7:InfFamR=3New}, $\Delta(r,3)=\gamma(r)-\Upsilon(r)=\cdot2^{(r-26)/3}(821-819)=2^{(r-23)/3}$.
 \end{proof}

In Tables \ref{tab7:LengFunrr33R=3} and \ref{tab7:LengFunr3464R=3}, the parameters $n,r$, and the covering density \eqref{eq11:CovDensity} of the best (as far as the authors know) binary linear covering $[n,n-r]_23$ codes of covering radius $R=3$ and redundancy (codimension) $3\le r\le64$  are written. The values of $n$ give an upper bound on the length function so that $\ell_2(r,3)\le n$; the star * notes the exact value of $\ell_2(r,3)$. The known results of \cite{DavCovNewConstrCovCod,KaikRoseADSlike2003,OstKaikUpBndBin1998} that are not presented in the book \cite {CHLL-bookCovCod} are noted by~$\bullet$. The new results, obtained in this paper using Theorems \ref{th7:PartitionOK}--\ref{th7:ell0=0R=3}, \ref{th7:r=26,38,41}, \ref{th7:newr=3t-1a}, of Section~\ref{sec7:R=3new} are written in bold font and also are noted by the big star $\bigstar$.

 \begin{table}[htbp]
\caption{The parameters of the best linear $[n,n-r]_23$ covering codes, $r\le33$,
such that $\ell_2(r,3)\le n$; the known values with the references seen in Section \ref{sec6:R=3known}; the new values in bold font and with $\bigstar$ seen in Section~\ref{sec7:R=3new}; $p(\ell)=p(\Hb,\ell)$; $\Delta=\Delta(r,3)$}
\centering
  \begin{tabular}
  {r|c|r|c|l|c|c|c|c}\hline
$r$&for $n$&\multicolumn{1}{c|}{$n$}&references&density&$p(0)$&$p(1)$&$p(2)$&$\Delta$\\\hline
3 &&$3^{*}$ &\cite{CHLL-bookCovCod,GrahSlo}&1&&&\\
4 &$\varphi(4)$&$5^{*}$ &\cite{CHLL-bookCovCod,DDL-IEEE,GrahSlo}&1.62500 &3&5&\\
5 &&$6^{*}$ &\cite{CHLL-bookCovCod,GrahSlo}          &1.31250 &&&\\
6 &&$7^{*}$ &\cite{CHLL-bookCovCod,GrahSlo}          &1       &7&&\\
7 &$\varphi(7)$&$11^{*}$&\cite{CHLL-bookCovCod,DDL-IEEE,GrahSlo}&1.81250&7&8& \\
8 &&$14^{*}$&\cite{CHLL-bookCovCod,GrahSlo}          &1.83594&10&& \\
9&&$\bullet18^{*}$&\cite{BaichVavLengthFunBin,BrPl1990,CHLL-bookCovCod,OstKaikUpBndBin1998}&1.92969&&$\mathbf{11\bigstar}$&\\
10&&22      &\cite{GrahSlo}                          &1.75195&&&  \\
11&&$23^{*}$&\cite{Golay1949,GrahSlo}                &1      &23&&  \\ \hline
12&&$\bullet37$&\cite{CHLL-bookCovCod,LobsPleRevisTab1994,OstKaikUpBndBin1998}&2.06885&&&\\
13&&$\bullet52$      &\cite{KaikRoseADSlike2003}              &2.86609&&&\\
14&&63&\cite{CHLL-bookCovCod,Dav90PIT,DougJanCovRadCalc1991}&2.54688&&&\\
15&$\psi(15)$&$\bullet71$      &\cite{DavCovNewConstrCovCod}&1.82227&&&32\\
16&$\varphi(16)$&95&\cite{DDL-IEEE}  &2.18164&&&\\
17&&126&\cite{DDL-IEEE}  &2.54442&&&\\
18&&153     &\cite{DDL-IEEE}  &2.27760&&&\\
19&&205     &\cite{DDL-IEEE}  &2.73900&&&\\
20&&$\bullet254$&\cite{DavCovNewConstrCovCod}            &2.60486&&&\\
21&&$\mathbf{303}$&$\bigstar$  &\textbf{2.21090}&&\textbf{35}&&\textbf{5}\\
22&$\varphi(22)$&383&\cite{DDL-IEEE}    &2.23254&&&\\
23&&511&\cite{Dav90PIT}  &2.65112&&&\\
24&&618     &\cite{DDL-IEEE}  &2.34477&&&\\
25&$\varphi(25)$&767&\cite{DDL-IEEE}  &2.24124&&&\\
26&$\mathbf{\Upsilon(26)}$&\textbf{818}&$\bigstar$&\textbf{1.35935}&\textbf{35}&\textbf{35}&\textbf{818}&$\mathbf{2}$\\
27&&1215    &\cite{DDL-IEEE}  &2.22725&&&&\\
28&$\varphi(28)$&1535&\cite{DDL-IEEE}    &2.24561&&&\\
29&&1642&\cite{DDL-IEEE}  &1.37436&&&\\
30&$\psi(30)$&$\bullet2303$    &\cite{DavCovNewConstrCovCod}&1.89597&&&\\
31&$\varphi(31)$&3071&\cite{DDL-IEEE}    &2.24780&&&\\
32&&3286&\cite{DDL-IEEE}  &1.37687&&&\\
33&$\psi(33)$&$\bullet4607$&\cite{DavCovNewConstrCovCod}&1.89720&&&\\\hline
 \end{tabular}
  \label{tab7:LengFunrr33R=3}
\end{table}

\begin{table}[htbp]
\caption{The parameters of the best linear $[n,n-r]_23$ covering codes, $34\le r\le64$,
such that $\ell_2(r,3)\le n$; the known values with the references (ref.) seen in Section \ref{sec6:R=3known}; the new values in bold font and with $\bigstar$ seen in Section~\ref{sec7:R=3new}; $p(\ell)=p(\Hb,\ell)$; $\Delta=\Delta(r,3)$}
\centering
  \begin{tabular}
  {r|c|r|c|l|c|c|c|c}\hline
$r$&for $n$&\multicolumn{1}{c|}{$n$}&ref.&density&$p(0)$&$p(1)$&$p(2)$&$\Delta$\\\hline
34&$\varphi(34)$&6143 &\cite{DDL-IEEE}                         &2.24890&&&\\
35&&6574 &\cite{DDL-IEEE}&1.37812&&&\\
36&$\psi(36)$&$\bullet9215$ &\cite{DavCovNewConstrCovCod}            &1.89782&&&\\
37&$\varphi(37)$&12287&\cite{DDL-IEEE}                         &2.24945&&&\\
38&$\mathbf{\widehat{\Upsilon}(38)}$&\textbf{13118}&$\bigstar$&\textbf{1.36871}&\textbf{57}&\textbf{59}&&$\mathbf{2^5}$\\
39&$\psi(39)$&$\bullet18431$&\cite{DavCovNewConstrCovCod}            &1.89813&&&\\
40&$\varphi(40)$&24575&\cite{DDL-IEEE}&2.24973&&&\\
41&$\mathbf{\widehat{\Upsilon}(41)}$&\textbf{26238}&$\bigstar$&\textbf{1.36902}&\textbf{89}&\textbf{91}&&$\mathbf{33}$\\
42&$\psi(42)$&$\bullet36863 $  &\cite{DavCovNewConstrCovCod}&1.89828&&&\\
43&$\varphi(43)$&49151	&\cite{DDL-IEEE}&2.24986&&&\\
44&$\mathbf{\Upsilon(44)}$&\textbf{52415}&$\bigstar$&\textbf{1.36426}&&$\mathbf{2^7+3}$&&$\mathbf{2^7}$\\
45&$\psi(45)$&$\bullet73727$   &\cite{DavCovNewConstrCovCod}&1.89836&&& \\
46&$\varphi(46)$&98303&\cite{DDL-IEEE}&	2.24993   &&&\\
47&$\mathbf{\Upsilon(47)}$   &\textbf{104831}&$\bigstar$&\textbf{1.36429}&&$\mathbf{2^8+3}$&&$\mathbf{2^8}$\\
48&$\psi(48)$&$\bullet147455$   &\cite{DavCovNewConstrCovCod}&1.89840&&&\\
49&$\varphi(49)$&196607&\cite{DDL-IEEE}   &2.24997&&&\\
50&$\mathbf{\Upsilon(50)}$   &\textbf{209663}&$\bigstar$&\textbf{1.36431}&&$\mathbf{2^9+3}$&&$\mathbf{2^9}$\\
51&$\psi(51)$&$\bullet294911$   &\cite{DavCovNewConstrCovCod}&1.89842&&&\\
52&$\varphi(52)$&393215&\cite{DDL-IEEE}   &2.24998&&&\\
53&$\mathbf{\Upsilon(53)}$   &\textbf{419327}&$\bigstar$&\textbf{1.36432}&&$\mathbf{2^{10}+3}$&&$\mathbf{2^{10}}$\\
54&$\psi(54)$&$\bullet589823$   &\cite{DavCovNewConstrCovCod}&1.89843&&&\\
55&$\varphi(55)$&786431&\cite{DDL-IEEE}   &2.24999&&&\\
56&$\mathbf{\Upsilon(56)}$& \textbf{838655}  &$\bigstar$&\textbf{1.36433}&&&\textbf{819}&$\mathbf{2^{11}}$\\
57&$\psi(57)$&$\bullet1179647$   &\cite{DavCovNewConstrCovCod}&1.89843&&&\\
58&$\varphi(58)$&1572863&\cite{DDL-IEEE}   &2.25000&&&\\
59&$\mathbf{\Upsilon(59)}$   &\textbf{1677311}&$\bigstar$&\textbf{1.36433}&&&\textbf{819}&$\mathbf{2^{12}}$\\
60&$\psi(60)$&$\bullet2359295$   &\cite{DavCovNewConstrCovCod}&1.89844&&&\\
61&$\varphi(61)$&3145727&\cite{DDL-IEEE}   &2.25000&&&\\
62&$\mathbf{\Upsilon(62)}$   &\textbf{3354623}&$\bigstar$&\textbf{1.36433}&&&\textbf{819}&$\mathbf{2^{13}}$   \\
63&$\psi(63)$&$\bullet4718591$   &\cite{DavCovNewConstrCovCod}&1.89844&&&   \\
64&$\varphi(64)$&6291455 &\cite{DDL-IEEE}  &2.25000&&&   \\ \hline
 \end{tabular}
  \label{tab7:LengFunr3464R=3}
\end{table}

In the tables, for the new codes we give the sizes $p(\Hb,\ell)$ of a $(3,\ell)$-partitions of a parity check matrix $\Hb$ (columns $p(\ell)$) and the decrease $\Delta(r,3)$ (column $\Delta$) of the known upper bounds on $\ell_2(r,3)$, provided by the new results of Section \ref{sec7:R=3new}.
Also, we write the new result $p(\Hb_{OK},1;\Ps_{OK2})=11$, obtained in Theorem \ref{th7:PartitionOK}.
The known and new values of $n$ are taken from Sections \ref{sec6:R=3known} and \ref{sec7:R=3new}. Also, for the known $n$ we give the references (the columns ``references'' and  ``ref.'').

\section{The known results on binary linear codes of covering radius 4}\label{sec8:R=4known}
As far as the authors know, the infinite code family of Theorem \ref{th8:KnownFamR=4} has the smallest asymptotic density $\overline{\mu}(R)$ \eqref{eq11:liminfdens} between the known infinite families of $[n,n-4t]_24$ codes.

\begin{theorem}\label{th8:DDLR=4} \emph{\cite{Dav90PIT}, \cite[Theorem 5.1]{DDL-IEEE}}
$\mathbf{Constructions~QM_1^4-QM_3^4}$. In Construction \emph{QM} of Section~$\ref{subsec32:QM}$, let the starting code $C_0$  be an $[n_0,n_0-r_0]_24,\ell_0$ code
with a parity check matrix $\Hb_0$ of the form \eqref{eq32:H0}. We define a new $[n,n-r]_2R_C,\ell$ code $C$ with $r=r_0+4m$, $n=2^mn_0+N_m$ by a parity check $r\times n$ matrix $\Hb_C$ of the form \eqref{eq32:QM-H}--\eqref{eq32:varD}. Let the indicator set $\Bs$, parameter $m$, and $r\times N_m$ submatrix $\D$ be chosen by one of the ways \eqref{eq8:ell0=2inp}--\eqref{eq8:ell0=3inp}  which set specific  Constructions \emph{QM}$_1^4$--\emph{QM}$_3^4$.
\begin{align}
&\emph{QM}_1^4: ~\ell_0=2,~\Bs\subseteq\F_{2^m},\,2^m\ge p(\Hb_0,2),\,
 \D=\D_2(4).\label{eq8:ell0=2inp}\db\\
 &\emph{QM}_2^4: ~\ell_0=2,\,c=\lfloor(2^m+2)/3\rfloor,\,c+1\ge p(\Hb_0,2),\label{eq8:ell0=2inp2}\db\\
 &\phantom{\emph{QM}_2^4: ~}\D\T{ and }\Bs\T{ are as $D_{13}^{4m}$ and $\beta$ in \emph{\cite[Theorem 5.1, Case 1iii]{DDL-IEEE}}}.\notag\db\\
&\emph{QM}_3^4: ~\ell_0=3,\,\Bs\subseteq\F_{2^m},\,2^m\ge p(\Hb_0,3),\,
 \D=\D_1(4).\label{eq8:ell0=3inp}
 \end{align}
Then the new code $C$ is an $[n, n - r]_24,\ell$ code with $r = r_0 + 4m$, and parameters:
\begin{align}
 &\emph{QM}_1^4: ~n=2^m(n_0+2)-2,~\ell=3,~p(\Hb_C,2)\le p(\Hb_0,2)+2,\label{eq8:ell0=2res}\db\\
 &\phantom{\emph{QM}_1^4: ~}p(\Hb_C,3)\le 3p(\Hb_0,2)+2\T{ for }m\ge3.\notag\db\\
 &\emph{QM}_2^4: ~n=2^m(n_0+2)-3,~\ell\ge2.\label{eq8:ell0=2res2}\db\\
 &\emph{QM}_3^4: ~n=2^m(n_0+1)-1,~\ell=3,~p(\Hb_C,3)\le p(\Hb_0,3)+1.\label{eq8:ell0=3res}
\end{align}
\end{theorem}

\begin{theorem}\label{th8:KnownFamR=4}
\emph{\cite[Examples 6,\,9]{DavCovNewConstrCovCod}}
\begin{description}
  \item[(i)] There exists a $[90,70]_24,2$ code with a parity check matrix $\Hb$ such that $p(\Hb,2)\le 25$.
  \item[(ii)] Let the asymptotic density $\overline{\mu}(R)$ be as in \eqref{eq11:liminfdens}. There exists an infinite family of $[n,n-r]_24$ codes with growing $r=4t$,  and  the following  parameters:
\begin{align}
&r=4t,~ r=20\T{ and }r\ge40,\,n=2944\cdot2^{r/4-10}-2=2944\cdot2^{t-10}-2,
\label{eq8:knownfam4t}\db\\
&\overline{\mu}(4)\thickapprox2.84669;~\ell_2(4t,4)\le2944\cdot2^{t-10}-2.\notag
 \end{align}
\end{description}
\end{theorem}

\begin{theorem}\label{th8:OKb code}\emph{\"{O}sterg{\aa}rd, Kaikkonen \cite[Table 2]{OstKaikUpBndBin1998}}
Let $C_{OK2}$ be the binary $[19,8]_2$ code with the parity check matrix
$\Hb_{OK2}=[\Ib_{11}\Mb_{OK2}]$, where $\Mb_{OK2}$ is the $11\times8$ binary matrix with the following columns in
hexadecimal notation:
\begin{align}\label{eq8:HOK}
&\Mb_{OK2}=[\mathrm{4EA, 771, 6, 86, 1CD, 3B4, 17E, 7AB}].
\end{align}
Then $C_{OK2}$ is a $[19,8]_24$ code of covering radius $4$.
\end{theorem}

\section{New results on binary linear codes of covering radius 4}\label{sec9:R=4new}
\begin{theorem}\label{th9:ell0=1}
$\mathbf{Construction~QM_4^4}$. In Construction \emph{QM} of Section~$\ref{subsec32:QM}$, let the starting code $C_0$  be an $[n_0,n_0-r_0]_24,1$ code with a parity check matrix $\Hb_0$ of the form \eqref{eq32:H0} admitting a $(4,1)$-partition $\Ps_0$
into $p(\Hb_0,1)$ subsets. We define a new $[n,n-r]_2R_C$ code $C$ by a parity check
$r\times n$ matrix $\Hb_C$ of the form \eqref{eq32:QM-H}--\eqref{eq32:Hcb2m}, where the indicator set $\Bs$, parameter $m$, and submatrix $\D$ are as follows:
\begin{equation}
\emph{QM}_4^4:~\Bs\subseteq\F_{2^m}^*,~2^m-1\ge p(\Hb_0,1),~m\T{ is odd},~
 \D=\D_5.\label{eq9:ell0=1inpR=4}
\end{equation}
Then the new code $C$ is an $[n, n - r]_24$ code of covering radius $4$ with the parameters:
\begin{equation}
\emph{QM}_4^4:~n=2^m(n_0+1)+n_{2m}-1,~r = r_0 + 4m.\label{eq9:ell0=1resR=4}
\end{equation}
\end{theorem}

\begin{proof}
 The values of $n$ and $r$ follow directly from construction.

 To prove $R_C=4$, we consider the representation as a sum of $\le4$ columns of $\Hb_C$ for an arbitrary column $\Ub=(\pib,\ub_1,\ub_2,\ub_3,\ub_4)^{tr}\in\F_2^{\,r}$, where $\pib\in\F_2^{\,r_0}$, $\ub_i\in\F_2^{\,m}$.
 \begin{description}
  \item[(1)] $\pib=\sum_{j=1}^v\hb_{i_j}$, $v=3,4$; $\hb_{i_j}$ belong to distinct subsets of $\Ps_0$; all $\beta_{i_j}$ are distinct.

  We find $\xb_j$ from the linear system $\sum_{j=1}^v\beta_{i_j}^{\delta-1}\xb_j=\ub_\delta$, $\delta=1,2,\ldots,v$.\\
  Now $\Ub=\sum_{j=1}^v (\hb_{i_j},\xb_j,\beta_{i_j} \xb_j,\beta_{i_j}^2 \xb_j,\beta_{i_j}^3 \xb_j)^{tr}+(4-v)(\zb_{r_0+3m},\ub_4+\sum_{j=1}^v \beta_{i_j}^3 \xb_j)^{tr}$.

  \item[(2)] $\pib=\sum_{j=1}^2\hb_{i_j}$; $\hb_{i_j}$ belong to distinct subsets of $\Ps_0$; $\beta_{i_1}\ne\beta_{i_2}$,
  $\beta_{i_1}^3\ne\beta_{i_2}^3$ as $m$ is odd.

  We find $\xb_j$ from the linear system $\sum_{j=1}^2\beta_{i_j}^{\delta-1}\xb_j=\ub_\delta$, $\delta=1,4$.\\
  Now $\Ub=\sum_{j=1}^2 (\hb_{i_j},\xb_j,\beta_{i_j} \xb_j,\beta_{i_j}^2 \xb_j,\beta_{i_j}^3 \xb_j)^{tr}+\boldsymbol{S}$, where $\boldsymbol{S}$ is a sum of $\le2$ columns of $\Hcb_{2m}$, see \eqref{eq32:varD2}.

  \item[(3)] $\pib=\hb_i$.

 $\Ub=(\hb_i,\ub_1,\beta_i\ub_1,\beta_i^2 \ub_1,\beta_i^3 \ub_1)^{tr}+\boldsymbol{T}$, where $\boldsymbol{T}$ is a sum of $\le3$ columns of $\D$.
\end{description}
Thus,  any column $\Ub \in\F_2^{\,r}$  can be represented as a sum of at most 4 columns of $\Hb_C$. This means that the covering radius of the new code $C$ is $R_C=4$.
\end{proof}

\begin{theorem}\label{th9:r=31,n=690}
There is a new $[690,659,3]_24$ code that provides the new upper bound on the length function $\ell_2(31,4)\le690$ and decreases the known bound $701$ of \emph{\cite{DDL-IEEE}} by $11$.
\end{theorem}

\begin{proof}
 We use Construction QM$_4^4$ of Theorem \ref{th9:ell0=1}. As the code $C_0$ we take the $[19,8]_24$ code $C_{OK2}$ of Theorem \ref{th8:OKb code} with the parity check matrix $\Hb_{OK2}=[\hb_1\hb_2\ldots\hb_{19}]=[\Ib_{11}\Mb_{OK2}]$, where $\Mb_{OK2}$ is given by \eqref{eq8:HOK}. We have $\hb_9+\hb_{10}+\hb_{14}=(000\, 0000\, 0100)^{tr}+(000\, 0000\, 0010)^{tr}+(000\, 0000\, 0110)^{tr}=\zb_{11}$. Thus, $C_{OK2}$ is a $[19,8,3]_24,1$ code, that implies $n_0=19$, $r_0=11$, $p(\Hb_{OK2},1)\le19$. This allows us to take odd $m=5$. As the
 $[n_{2m},n_{2m}-2m]_22$ code $\Vc_{2m}$, see \eqref{eq32:Hcb2m}, we take the $[51,41,3]_22$ code \cite{KaikRoseADSlike2003},
 see Theorems \ref{th4:KR code}, \ref{th5:PartitionKR}. By \eqref{eq9:ell0=1resR=4}, we obtain a new $[690,690-31]_24$ code.
\end{proof}

\begin{theorem}\label{th9:newFamR=4}
Let the asymptotic covering density $\overline{\mu}(R)$ be as in \eqref{eq11:liminfdens}. There exists a new infinite family of binary linear  $[n,n-r,3]_24,3$ codes with a parity check matrix $\Hb$ of the form \eqref{eq32:QM-H}--\eqref{eq32:Wbm}, \eqref{eq8:ell0=3inp}, with growing $r=4t$ and  the parameters as in \eqref{eq9:newfam4t}:
\begin{align}
&r=4t,~ r=40\T{ and }r\ge68,\,n=2943\cdot2^{r/4-10}-1=2943\cdot2^{t-10}-1,
\label{eq9:newfam4t}\db\\
&\overline{\mu}(4)\thickapprox2.84282;~\ell_2(4t,4)\le2943\cdot2^{t-10}-1,~\Delta(r,4)=2^{t-10}-1.\notag
 \end{align}
\end{theorem}

\begin{proof}
  In Construction QM$_1^4$ of Theorem \ref{th8:DDLR=4}, we take as the starting code $C_0$ the\\
   $[90,70]_24,2$ code of Theorem \ref{th8:KnownFamR=4}(i) with $n_0=90$, $r_0=20$, $\ell_0=2$, and a parity check matrix $\Hb_0$ such that $p(\Hb_0,2)\le 25$. By \eqref{eq8:ell0=2inp}, we can take $m=5$. By \eqref{eq8:ell0=2res}, we obtain a $[2942,2902]_24,3$ code $C$ with $p(\Hb_C,3)\le 77$. We take this code as $C_0$ for Construction QM$_3^4$ of Theorem \ref{th8:DDLR=4} and put $m\ge7$ according to \eqref{eq8:ell0=3inp}. As a result, we obtain the needed new infinite code family. The value of $\Delta(r,4)$ follows from \eqref{eq8:knownfam4t} and \eqref{eq9:newfam4t}.
\end{proof}

\begin{theorem}\label{th9:r=48..64}
  For $r\ge48$, there are new $[n,n-r,3]_24,\ell$ codes with the parameters:
  \begin{equation}
  R=4,~r=4t\ge48,~n=2944\cdot2^{r/4-10}-3,~\ell\ge2;~\ell_2(r,4)\le n, ~\Delta(r,4)=1.\label{eq9:-3a}
  \end{equation}
 \end{theorem}

\begin{proof}
 In Construction QM$_2^4$ of Theorem \ref{th8:DDLR=4}, we take as the starting code $C_0$ the\\ $[90,70]_24,2$ code of Theorem \ref{th8:KnownFamR=4}(i) with $n_0=90$, $r_0=20$, $\ell_0=2$, and a parity check matrix $\Hb_0$ such that $p(\Hb_0,2)\le 25$. By \eqref{eq8:ell0=2inp2}, we can take $m\ge7$.  As a result, we obtain the needed new codes. The value of $\Delta(r,4)$ follows from \eqref{eq8:knownfam4t} and \eqref{eq9:-3a}.
\end{proof}

\section{Conclusions and open problems}\label{sec10:conclus} For binary covering codes, we have improved the results of \cite{DDL-ACCT2-1990,DDL-ACCT3-1992,DDL-IEEE} and \cite{DavCovNewConstrCovCod} obtained in 1990--1994 (more than 30 years ago) and 2001 (25 years ago) that were not surpassed by anybody since then.

Using versions of $q^m$-concatenating constructions, a part of which is proposed in this paper, we obtained new infinite families of binary linear $[n,n-r]_2R$ codes of length $n$, codimension (redundancy) $r$ and covering radius $R$\/ for $R=2,r=2t$, and $R=3,r=3t-1$, and $R=4,r=4t$, where $t$ is growing. In the construction of the new infinite families, an important role play new partitions of the column sets of parity check matrices of some codes (in particular, the $[51,41]_22$ code \cite{KaikRoseADSlike2003}) into subsets with special properties. The asymptotic covering densities provided by the codes of the new families  are smaller (i.e. better) than the known ones. The lengths of the new codes give new constructive upper bounds on the length function $\ell_2(r,R)$ that is the minimum possible length of binary linear code of codimension  r and covering radius R. These bounds are smaller (i.e. better) than the known ones.

The new infinite code families are obtained in Theorems \ref{th5:r=18,28}, \ref{th5:R=2NewInfFam}, \ref{th7:r=26,38,41}, \ref{th7:newr=3t-1a}, \ref{th9:newFamR=4}; their parameters are collected in Theorem \ref{th2:InfFamNew}.
Also, in Theorems \ref{th5:-3r=22,24,26}, \ref{th7:303R=3}, \ref{th7:r=26,38,41}, \ref{th9:r=31,n=690}, \ref{th9:r=48..64} we obtained  sporadic new covering codes, whose parameters are collected in  Theorem \ref{th2:SporasNew}. The new versions of the $q^m$-concatenating constructions are given in Theorems \ref{th5:-3}, \ref{th5:r=18,28}, \ref{th5:n0Psi}, \ref{th7:ell0=0R=3}, \ref{th9:ell0=1}.

The approaches and methods of this paper can be used to obtain new results for codes of covering radius $R\ge 5$. As far as it is known to the authors, in the literature, only sporadic examples of covering codes with $R\ge 5$ are given  \cite{CHLL-bookCovCod,DavCovNewConstrCovCod,DDL-IEEE,KaikRoseADSlike2003,OstKaikUpBndBin1998} whereas infinite families are not presented.

\textbf{Open problems.} $\bullet$ To obtain new infinite families of binary linear $[n,n-r]_2R$ codes improving the known results of \cite{DDL-ACCT2-1990,DDL-ACCT3-1992,DDL-IEEE,DavCovNewConstrCovCod,GabDavTombR=2} for the parameters: $R=2,r=2t-1$; $R=3,r=3t-2, r=3t$; and $R=4,r=4t-3,r=4t-2,r=4t-1$,  where  $t$ is growing.

$\bullet$ To obtain infinite families of binary linear $[n,n-r]_2R$ codes of covering radius $R\ge 5$ with relatively good parameters.

$\bullet$ To improve sporadic known binary linear covering codes that are not included in infinite families.

\section*{Acknowledgments} The research of S. Marcugini and F. Pambianco was supported in part by the Italian
National Group for Algebraic and Geometric Structures and their Applications (GNSAGA -
INDAM) (Contract No. U-UFMBAZ-2019-000160, 11.02.2019) and by University of Perugia
(Project No. 98751: Strutture Geometriche, Combinatoria e loro Applicazioni, Base Research
Fund 2017--2019; Fighting Cybercrime with OSINT, Research Fund 2021). This work was partially funded by the SERICS project (PE00000014) under the MUR National Recovery and Resilience Plan funded by the European Union-NextGenerationEU.


\begin{thebibliography}{99}
\bibitem{Hats-on-line}
S. Aravamuthan, S. Lodha,
Covering codes for Hats-on-a-line, \emph{Electronic J. Combinatorics}, \textbf{13}, Article 21 (2006)
\url{https://doi.org/10.37236/1047}

\bibitem{Bulev}
J.T. Astola, R.S. Stankovi\'{c},
Determination of sparse representations of multiple-valued logic functions by using covering codes,
\emph{J. Mult.-Valued Logic Soft Comput.}, \textbf{19}(4), 285--306 (2012)\\
\url{https://www.oldcitypublishing.com/journals/mvlsc-home/mvlsc-issue-contents/mvlsc-volume-19-number-4-2012/mvlsc-19-4-p-285-306}

\bibitem{BaichBoy2008dim6}
T. Baicheva, I. Bouyukliev, On the least covering radius of the binary linear codes of dimension 6,
\emph{Adv. Math. Commun.}, \textbf{4}(3), 399--404 (2010)\\
\url{https://doi.org/10.3934/amc.2010.4.399}


\bibitem{BaichVavLengthFunBin}
T. Baicheva, V. Vavrek, On the least covering radius of binary linear codes
with small lengths, \emph{IEEE Trans. Inform. Theory}, \textbf{49}(3), 738--740 (2003)\\
\url{https://doi.org/10.1109/TIT.2002.808099}
%

\bibitem{BDGMP-MultCovFurth}
D. Bartoli, A.A. Davydov, M. Giulietti, S. Marcugini, F. Pambianco,
Further results on multiple coverings of the farthest-off points,
\emph{Adv. Math. Commun.}, \textbf{10}(3), 613--632 (2016)
\url{https://doi.org/10.3934/amc.2016030}


\bibitem{libroBierbrauer}
J. Bierbrauer, \emph{Introduction to Coding Theory}, $2^{nd}$ edition,
Chapman \& Hall, CRC Press, Boca Raton, Taylor \& Francis Group, Boca Raton, (2017)

\bibitem{BierbStegan}
J. Bierbrauer, J. Fridrich, Constructing Good Covering Codes for Applications in Steganography. In:
\emph{Lecture Notes in Computer Science, Trans. Data Hiding Multimedia Security
III (ed. Y.Q. Shi)}, Springer-Verlag, \textbf{4920}, 1--22 (2008)\\
 \url{https://doi.org/10.1007/978-3-540-69019-1_1}

\bibitem{Bo-Sz-Ti}
  E. Boros, T. Sz\H{o}nyi, K. Tichler,
   On defining sets for projective planes,
  \emph{Discrete Math.}, \textbf{303}(1--3), 17-31 (2005) \url{https://doi.org/10.1016/j.disc.2004.12.015}

 \bibitem{Magma} W. Bosma, J. Cannon, C. Playoust, The Magma algebra system. I.
 The user language, J. Symbolic Comput. \textbf{24}(3-4), 235--265 (1997)\\
 \url{https://doi.org/10.1006/jsco.1996.0125}

\bibitem{Handbook-coverings}
R.A. Brualdi, S. Litsyn, V. Pless, Covering radius. In: V.S. Pless,
 W.C. Huffman (eds.) \emph{Handbook of Coding Theory}, vol.~1, pp. 755--826.
  Elsevier, Amsterdam (1998)

\bibitem{BrPl1990}
R.A. Brualdi, V.S. Pless, On the length of codes with a given
covering radius, in Coding Theory and Design Theory. Part I.
Coding Theory.  The IMA Volumes in Mathematics and Its Applications, vol 20. Springer,
New York: Springer-Verlag, 1990, pp. 9-15.
\url{https://doi.org/10.1007/978-1-4613-8994-1_2}

\bibitem{BrPlWi}
R.A. Brualdi, V.S. Pless, R.M.
    Wilson, Short codes with a given covering radius,
    \emph{IEEE Trans. Inform. Theory} \textbf{35}(1), 99--109 (1989)\\
\url{https://doi.org/10.1109/18.42181}

\bibitem{ListDecCovCod-2022}
H. Chen, List-decodable codes and covering codes,
arXiv:2109.02818 [cs IT]  (2021)  \url{https://doi.org/10.48550/arXiv.2109.02818}

\bibitem{Graismer-2024}
H. Chen, L. Qu, C. Li, S. Lyu, L. Xu, M. Zhou, Generalized Singleton type upper bounds,
\emph{IEEE Trans. Inform. Theory}, \textbf{70}(5), 3298--3308 (2024)\\ \url{https://doi.org/10.1109/TIT.2023.3346179}

\bibitem{ChengTang-FactDesigns}
C-S. Cheng, B. Tang, \emph{Theory of Nonregular Factorial Designs}, Chapman and Hall/CRC, New York (2025)
\url{https://doi.org/10.1201/9781003371861}

\bibitem{CohenFranklBounds}
G. Cohen, P. Frankl, Good coverings of Hamming spaces with
spheres, \emph{Discr. Math.}, \textbf{56}(2-3),  125--131 (1985)
\url{https://doi.org/10.1016/0012-365X(85)90020-2}

\bibitem{CHLL-bookCovCod}
G. Cohen, I. Honkala, S. Litsyn, A. Lobstein, \emph{Covering Codes}, North-Holland
  Mathematical Library, vol.~54. Elsevier, Amsterdam, The Netherlands (1997)

\bibitem{CKMS-survey1985}
G.D. Cohen, M.G. Karpovsky, H.F. Mattson, Jr., J.R. Schatz, Covering radius -- Survey and
recent results, \emph{IEEE Trans. Inform. Theory}, \textbf{31}(3), 328-343 (1985)\\
\url{https://doi.org/10.1109/TIT.1985.1057043}

\bibitem{CLLM-survey1985-1994}
G.D. Cohen, S.N. Litsyn, A.C. Lobstein, H.F. Mattson, Jr.,
Covering Radius 1985--1994, \emph{Applic. Algebra  Engeneer. Commun. Comput. AAECC}, \textbf{8}(3), 173--239 (1997)
\url{https://doi.org/10.1007/s002000050061}

\bibitem{CLS-1986}
G.D. Cohen, A.C. Lobstein, N.J.A. Sloane, Further results on the covering radius of codes, \emph{IEEE
Trans. Inform. Theory}, \textbf{32}(5), 680--694 (1986)\\
\url{https://doi.org/10.1109/TIT.1986.1057227}

\bibitem{Dav90PIT}  A.A. Davydov, Construction of linear covering codes, \emph{Problems Inform. Transmiss.} \textbf{26}, 317--331
(1990)  \url{https://www.researchgate.net/publication/268619798}

\bibitem{DavParis}
A.A. Davydov,  Construction of codes with covering radius 2.  In: G. Cohen, A. Lobstein, G. Zemor, S. Litsyn (eds.) \emph{Algebraic Coding (Algebraic Coding 1991) Lect. Notes Comput. Science},
\textbf{573},  23--31. Springer--Verlag, Berlin, Heidelberg (1992)\\
\url{https://doi.org/10.1007/BFb0034337}

\bibitem{Dav95}
 A.A. Davydov,
 Constructions and families of covering codes and saturated sets of points in projective geometry,
\emph{IEEE Trans. Inform. Theory}, \textbf{41}(6), 2071--2080 (1995) \url{https://doi.org/10.1109/18.476339}

\bibitem{DavCovNewConstrCovCod}
A.A. Davydov, New constructions of covering codes, \emph{Des. Codes Cryptogr.}, \textbf{22}(3), 305–316 (2001) \url{https://doi.org/10.1023/A:1008302507816}

\bibitem{DDL-ACCT2-1990}
A.A. Davydov, A.Yu. Drozhzhina-Labinskaya, Binary linear
codes with covering radii 3 and 4. In: \emph{Proc. 2nd International Workshop on Algebraic
and Combinatorial Coding Theory -- ACCT-2}, pp. 56-57, Leningrad, 1990. Available from: \\
\url{https://www.researchgate.net/publication/397051363}

\bibitem{DDL-ACCT3-1992}
A.A. Davydov, A.Yu. Drozhzhina-Labinskaya, Constructions
of binary linear covering codes. In: \emph{Proc. 3rd International Workshop on Algebraic
and Combinatorial Coding Theory -- ACCT-3}, pp. 51-54, Voneshta Voda, Bulgaria, 1992. Available from: \\
\url{https://www.researchgate.net/publication/397065147}


\bibitem{DDL-IEEE}
A.A. Davydov, A.Yu. Drozhzhina-Labinskaya, Constructions, families, and tables of binary linear covering codes, \emph{IEEE Trans. Inform. Theory} \textbf{40}(4), 1270--1279 (1994)
\url{https://doi.org/10.1109/18.335937}

\bibitem{DGMP-AMC}
A.A. Davydov, M. Giulietti, S. Marcugini, F.~Pambianco,  Linear nonbinary
  covering codes and saturating sets in projective spaces, \emph{Adv.
  Math. Commun.}  \textbf{5}(1),  119--147 (2011) \url{https://doi.org/10.3934/amc.2011.5.119}

\bibitem{Delsarte} P. Delsarte, Four fundamental parameters of a code
and their combinatorial significance, \emph{Inform. Control}, \textbf{23}(5), 407--438 (1973)\\
 \url{https://doi.org/10.1016/S0019-9958(73)80007-5}

\bibitem{DougJanCovRadCalc1991}
R. Dougherty, H. Janwa, Covering radius computations for binary cyclic codes, \emph{Math. Comput.}, \textbf{57}(195), 415--434 (1991) \url{https://www.jstor.org/stable/i352329}

\bibitem{EtzStorm2016}
T. Etzion, L. Storme,
 Galois geometries and coding theory,
\emph{ Des. Codes Cryptogr.},  \textbf{78}, 311--350 (2016) \url{https://doi.org/10.1007/s10623-015-0156-5}

\bibitem{GabDavTombR=2}
E.M. Gabidulin, A.A. Davydov, L.M. Tombak, Linear codes with covering radius 2 and other
new covering codes, \emph{IEEE Trans. Inform. Theory}, \textbf{37}(1), 219--224 (1991)\\
\url{https://doi.org/10.1109/18.61146}

\bibitem{GalKaba} F. Galand, G. Kabatiansky,
Information hiding by coverings. In:
\emph{Proc. 2003 IEEE Inform. Theory Workshop}, Paris,  France, pp. 151--154 (2003)\\
\url{https://doi.org/10.1109/ITW.2003.1216717}

\bibitem{GalKabaPIT}
F. Galand, G. Kabatiansky,
Coverings, centered codes, and combinatorial steganography,
 \emph{Probl. Inform. Transmission}, \textbf{45}(3), 289--294 (2009)\\
 \url{https://doi.org/10.1134/S0032946009030107}

\bibitem{Cayley}
 C. Garcia, C. Peyrat,
 Large Cayley graphs on an abelian group, \emph{Discrete Appl. Math.}, \textbf{75}, 125--133 (1997)
\url{https://doi.org/10.1016/S0166-218X(96)00084-4}


\bibitem{Giul2013Survey}
M. Giulietti,  The geometry of covering codes: small complete caps and saturating sets in Galois spaces,
in \emph{Surveys in Combinatorics 2013}, (eds. S. R. Blackburn, R. Holloway,  M. Wildon) London Math. Soc., Lecture
  Note Series, \textbf{409}, Cambridge University Press, Cambridge,
 51--90 (2013)\\
   \url{https://doi.org/10.1017/CBO9781139506748.003}

\bibitem{Golay1949}
M.J.E. Golay, Notes on digital coding,  \emph{Proc. IRE}, \textbf{37}(6), p. 657, (1949)\\
\url{https://doi.org/10.1109/JRPROC.1949.233620}

\bibitem{GrahSlo}
R.L. Graham, N.J.A. Sloane, On the covering radius of codes,
\emph{IEEE Trans. Inform. Theory}, \textbf{31}(3), 385--401 (1985) \url{https://doi.org/10.1109/TIT.1985.1057039}

\bibitem{LPN-CovCod}
Q. Guo, T. Johansson, C. L\"{o}ndahl,
Solving LPN using covering codes,
\emph{J. Cryptology}, \textbf{33}(1), 1--33 (2020)
\url{https://doi.org/10.1007/s00145-019-09338-8}

\bibitem{HamalHonLitsOstFootbPool}
H.O. H\"am\"al\"ainen, I. Honkala, S. Litsyn, P. \"{O}sterg{\aa}rd, Football pools—a game for mathematicians, \emph{American Math. Monthly}, \textbf{102}(7), 579--588 (1995)\\
\url{https://doi.org/10.2307/2974552}

\bibitem{HamalFootbPool}
H.O. H\"am\"al\"ainen, S. Rankinen,
Upper bounds for football pool problems and mixed covering codes,
\emph{J. Combinatorial Theory Ser. A.}, \textbf{56}(1), 84--95 (1991)\\
\url{https://doi.org/10.1016/0097-3165(91)90024-B}

\bibitem{Hirs}
 J.W.P. Hirschfeld,
\emph{Projective Geometries over Finite Fields},
2$^{nd}$ edition, Oxford Univ. Press, Oxford, 1999.

\bibitem{HirsStor-2001}
J.W.P. Hirschfeld, L. Storme,
The  packing problem in statistics, coding theory and finite projective spaces: Update 2001,
 in (eds. A. Blokhuis, J.W.P. Hirschfeld, D. Jungnickel, J.A. Thas), \emph{Finite
    Geometries (Proc. 4th Isle of Thorns Conf., July 16-21, 2000)},
    Develop. Math.,  \textbf{3}, Kluwer, Dordrecht, 2001, 201--246\\
    \url{https://doi.org/10.1007/978-1-4613-0283-4_13}

\bibitem{PartSumQuer}
 C.T. Ho, J. Bruck, R. Agrawal,
 Partial-sum queries in OLAP data cubes using covering codes,
 \emph{IEEE Trans. Computers}, \textbf{47}(12), 1326--1340 (1998)\\
  \url{https://doi.org/10.1109/12.737680}

\bibitem{Hon-length} I.S. Honkala, On lengthening of
    covering codes, \emph{Discrete Math.}, \textbf{106--107}, 291--295 (1992)
\url{https://doi.org/10.1016/0012-365X(92)90556-U}

\bibitem{HonkLits1996}
I. Honkala, S. Litsyn, Generalizations of the covering radius problem in coding theory,
\emph{Bull. Inst. Combinatorics its Applications}, \textbf{17}(1). 39--46 (1996) Available from: \\
\url{https://www.researchgate.net/publication/2278556}

\bibitem{HufPless}  W.C. Huffman, V.S. Pless, \emph{Fundamentals of Error-Correcting Codes}, Cambridge University Press, (2003)

\bibitem{Janwa} H. Janwa,  Some optimal codes from algebraic geometry and their
covering radii, \emph{Europ. J. Combin.} \textbf{11}(3),  249--266 (1990)\\
\url{https://doi.org/10.1016/S0195-6698(13)80125-4}

\bibitem{KabPan} G.A. Kabatyansky, V.I. Panchenko,
Unit sphere packings and coverings of the Hamming
    space, \emph{Problems Inform. Transmiss.}, \textbf{24}(4),
    261--272 (1988)

\bibitem{KaikRoseADSlike2003}
M.K. Kaikkonen, P. Rosendahl, New covering codes from an ADS-like construction,
\emph{IEEE Trans. Inform. Theory}, \textbf{49}(7), 1809--1812 (2003)\\
\url{https://doi.org/10.1109/TIT.2003.813508}

\bibitem{KaipaDeepHoles}
K. Kaipa, Deep holes and MDS extensions of Reed-Solomon codes,
 \emph{IEEE TTrans. Inform. Theory}, \textbf{63}(8), 4940-–4948 (2017)
\url{https://doi.org/10.1109/TIT.2017.2706677}

\bibitem{KilbySloane1987-II}
K.E. Kilby, N.J.A. Sloane, On the covering problem for codes II.
Codes of low dimension; normal and abnormal codes,
\emph{SIAM J. Alg. Discrete Methods}, \textbf{8}(4), 619--627 (1987)
\url{https://doi.org/10.1137/0608050}

\bibitem{KKKPS}
 G. Kiss, I. Kov\'{a}cs, K. Kutnar, J. Ruff, P. \v{S}parl,
A note on a geometric construction of large Cayley graphs of given degree and diameter,
Studia Univ. ``Babe\c{s}--Bolyai'', Mathematica,
\textbf{LIV}(3) (2009), 77--84.
Available from:\\ \url{https://www.cs.ubbcluj.ro/journal/studia-mathematica/archive/2009-3/kiss.pdf}

\bibitem{Klein-Stor}
A. Klein, L. Storme,
 Applications of finite geometry in coding theory and cryptography,
in \emph{Security, Coding
  Theory and Related Combinatorics, NATO Science for Peace and Security, Ser. -
  D: Information and Communication Security}, (eds. D. Crnkovi\'{c}, V. Tonchev) \textbf{29}, IOS Press
  (2011), 38--58\\ \url{https://doi.org/10.3233/978-1-60750-663-8-38}

\bibitem{swithc}
C.E. Koksal, An analysis of blocking switches using  error control codes,
\emph{IEEE Trans. Inform. Theory}, \textbf{53}(8), 2892--2897 (2007)
\url{https://doi.org/10.1109/TIT.2007.901191}

\bibitem{KrivSud-CovCodImprovDens}
M. Krivelevich, B. Sudakov, V.H. Vu , Covering codes with improved density,
\emph{IEEE Trans. Inform. Theory}, \textbf{49}(7), 1812--1815 (2003)\\
\url{https://doi.org/10.1109/TIT.2003.813490}

\bibitem{identif2}
T.K. Laihonen, Sequences of optimal identifying codes,
\emph{IEEE Trans. Inform. Theory}, \textbf{48}(3), 774--776 (2002)
\url{https://doi.org/10.1109/18.986043}

\bibitem{LandSt}
I. Landjev, L. Storme,
Galois geometry and coding theory,
in \emph{Current Research Topics in Galois geometry,} (eds. J. De Beule,
  L. Storme), Chapter 8, NOVA Academic, New York, 2011, 187--214.  Available from:\\
\url{https://www.researchgate.net/publication/216667309}



\bibitem{LenSerHat2002}
H.W. Lenstra, Jr., G. Seroussi,
On hats and other covers.
In: \emph{Proc. IEEE Int. Symp. Inform. Theory}, IEEE Xplore, Lausanne,
Switzerland, July 2002, 342--342 (2002)
\url{https://doi.org/10.1109/ISIT.2002.1023614}

\bibitem{LobstBibl}
A. Lobstein,  Covering radius, an online bibliography (2019),\\
  \url{https://www.lri.fr/~lobstein/bib-a-jour.pdf}

\bibitem{LobsPleRevisTab1994}
A. Lobstein, V. Pless, The length function: A revised table, in: G. Cohen, S. Litsyn, A. Lobstein,
G. Zemor, eds., Algebraic Coding; Proc. 1st French-Israeli Workshop Paris, Algebraic Coding 1993, Lecture Notes
in Computer Science, vol. 781, Springer, Berlin, 1994, 51--55.
\url{https://doi.org/10.1007/3-540-57843-9_7}

\bibitem{MWS} F.J. MacWilliams, N.J.A. Sloane, \emph{The Theory of Error-Correcting Codes}, $3^{rd}$ edition, North-
Holland, Amsterdam, The Netherlands (1981)

\bibitem{MenNewCrypCovCod}
B. Mennink, S. Neves,
On the resilience of Even-Mansour to invariant permutations,
\emph{Des. Codes Cryptogr.}, \textbf{89}(5), 859--893 (2021)\\
\url{https://doi.org/10.1007/s10623-021-00850-2}

\bibitem{NavSamor-LPbound&CovArgum2009}
M. Navon, A. Samorodnitsky, Linear programming bounds for codes via a covering argument, \emph{Discrete Comput. Geom.}, \textbf{41}(2), 199--207 (2009)\\
 \url{https://doi.org/10.1007/s00454-008-9128-0}

\bibitem{OstKaikUpBndBin1998}
P.R.J. \"{O}sterg{\aa}rd, M.K. Kaikkonen, New upper bounds for
binary covering codes,\emph{ Discrete Math.}, \textbf{178}(1--3), 165--179 (1998)\\
\url{https://doi.org/10.1016/S0012-365X(97)81825-0}

\bibitem{Roth}
 R.M. Roth, Introduction to Coding Theory, Cambridge, Cambridge Univ. Press, 2006.

\bibitem{StanAstMor-book}
R.S. Stankovi\'{c}, J.T. Astola, C. Moraga, \emph{Representation of Multiple-Valued Logic Functions (Synthesis Lectures on Digital Circuits $\&$ Systems (SLDCS))},
Springer (2012)
\url{https://dx.doi.org/10.2200/S00420ED1V01Y201205DCS037}

\bibitem{Struik1994R=2-3Bin}
R. Struik, On the structure of linear codes with covering radius two and three
\emph{IEEE Trans. Inform. Theory}, \textbf{40}(5), 1406--1416 (1994)
\url{https://doi.org/10.1109/18.333857}

\bibitem{Struik}
R. Struik, Covering Codes, Ph.D thesis, Eindhoven University of Technology, The Netherlands, 1994.
 \url{https://doi.org/10.6100/IR425174}

\bibitem{PlanFuncCovCodWuPan2025}
Y. Wu, Y. Pan, Linear codes from planar functions and related covering codes, \emph{Finite Fields Their Applications},
\textbf{101}, paper 102535 (2025)\\
\url{https://doi.org/10.1016/j.ffa.2024.102535}

\end{thebibliography}
\end{document}